\documentclass[reqno]{amsart}

\usepackage{lmodern}
\usepackage[T1]{fontenc}
\usepackage{textcomp}

\usepackage[utf8]{inputenc}
\usepackage[unicode]{hyperref}
\usepackage{amsfonts}
\usepackage{indentfirst}
\usepackage[pdftex]{graphicx}
\usepackage{enumerate}
\usepackage{amssymb}
\usepackage{amsmath}
\usepackage{mathrsfs}
\usepackage{paralist}
\usepackage{amsmath,amscd}
\usepackage{amsthm}
\usepackage{bbold}
\usepackage{xcolor}
\usepackage{accents}

\newcommand{\tieconcat}{%
	\mathbin{\mathpalette\dotieconcat\relax}%
}
\newcommand{\dotieconcat}[2]{
	\text{\raisebox{.8ex}{$\smallfrown$}}%
}

\usepackage{scalerel,stackengine}
\stackMath
\newcommand\reallywidecheck[1]{%
	\savestack{\tmpbox}{\stretchto{%
			\scaleto{%
				\scalerel*[\widthof{\ensuremath{#1}}]{\kern-.6pt\bigwedge\kern-.6pt}%
				{\rule[-\textheight/2]{1ex}{\textheight}}
			}{\textheight}%
		}{0.5ex}}%
	\stackon[1pt]{#1}{\scalebox{-1}{\tmpbox}}%
}

\DeclareMathOperator{\dom}{dom}
\DeclareMathOperator{\ran}{ran}
\DeclareMathOperator{\cf}{cf}
\DeclareMathOperator{\supp}{supp}
\DeclareMathOperator{\pr}{pr}
\DeclareMathOperator{\htt}{ht}

\DeclareMathOperator{\Le}{L}
\DeclareMathOperator{\mo}{mod}

\DeclareMathOperator{\otp}{otp}

\newtheorem{ttt}{Theorem}[section]
\newtheorem{llll}[ttt]{Lemma}
\newtheorem{ccc}[ttt]{Claim}
\newtheorem{eee}[ttt]{Example}
\newtheorem{fff}[ttt]{Fact}
\newtheorem{rrr}[ttt]{Remark}
\newtheorem{sss}[ttt]{Statement}
\newtheorem{ddd}[ttt]{Definition}
\newtheorem{qqq}[ttt]{Question}
\newtheorem{cccc}[ttt]{Corollary}
\newtheorem{nnn}[ttt]{Notation}
\newtheorem{ppp}[ttt]{Problem}
\newtheorem{pppp}[ttt]{Proposition}
\newtheorem{ccccc}[ttt]{Conjecture}
\newtheorem{sccc}[ttt]{Subclaim}

\newcommand{\beq}{\begin{equation} }

\newcommand{\bt}{\begin{ttt}}
\newcommand{\bl}{\begin{llll}}
\newcommand{\bc}{\begin{ccc}}
\newcommand{\bex}{\begin{eee}}
\newcommand{\bfa}{\begin{fff}}
\newcommand{\br}{\begin{rrr}\upshape}
\newcommand{\bst}{\begin{sss}}
\newcommand{\bd}{\begin{ddd}\upshape}
\newcommand{\bdd}{\begin{ddd}}

\newcommand{\bq}{\begin{qqq}}
\newcommand{\bnn}{\begin{nnn}}
\newcommand{\bpr}{\begin{ppp}}
\newcommand{\bprop}{\begin{pppp}}
\newcommand{\bcor}{\begin{cccc}}
\newcommand{\bcon}{\begin{ccccc}}

\newcommand{\bsc}{\begin{sccc}}
\newcommand{\esc}{\end{sccc}}

\newcommand{\eeq}{\end{equation}}
\newcommand{\et}{\end{ttt}}
\newcommand{\el}{\end{llll}}
\newcommand{\ec}{\end{ccc}}
\newcommand{\eex}{\end{eee}}
\newcommand{\efa}{\end{fff}}
\newcommand{\er}{\end{rrr}}
\newcommand{\est}{\end{sss}}
\newcommand{\ed}{\end{ddd}}
\newcommand{\eq}{\end{qqq}}
\newcommand{\ecor}{\end{cccc}}
\newcommand{\econ}{\end{ccccc}}
\newcommand{\enn}{\end{nnn}}
\newcommand{\epr}{\end{ppp}}
\newcommand{\eprop}{\end{pppp}}

\newcommand{\bp}{\begin{proof}}
\newcommand{\ep}{\end{proof}}

\newcommand{\es}{\emptyset}
\newcommand{\beeq}{\begin{equation}}
\newcommand{\eeeq}{\end{equation}}

\def\sq{\subseteq}
\def\se{\setminus}

\def\restricttt#1{\raise-.3ex\hbox{\ensuremath|}_{#1}}
\def\restrictt#1{\raise-.2ex \hbox{\ensuremath \upharpoonright }{#1}}

\def\restri#1{\raise-.1ex\hbox{\ensuremath \upharpoonright}_{#1}}

\newcommand\restrict[1]{\raisebox{-.3ex}{$\upharpoonright$}_{#1}}

\def\name#1{\mathchoice%
	{\setbox0=\hbox{$\displaystyle #1$}
		\setbox1=\vtop{\ialign{##\crcr
				$\hfil{\displaystyle #1}\hfil$\crcr\noalign{\kern2pt%
					\nointerlineskip}$\hfil\mathord{\displaystyle \sim}%
				\hfil$\crcr\noalign{\kern2pt\nointerlineskip}}}%
		\vphantom{\copy1}%
		\setbox2=\hbox{$\displaystyle \sim$}%
		\wd1=\wd0\dp1=0cm\ifdim\wd2>\wd1 \wd1=\wd2\else\relax\fi
		\ht1=\ht0\relax
		\box1}%
	{\setbox0=\hbox{$\textstyle #1$}
		\setbox1=\vtop{\ialign{##\crcr
				$\hfil{\textstyle #1}\hfil$\crcr\noalign{\kern1.2pt%
					\nointerlineskip}$\hfil\mathord{\textstyle \sim}%
				\hfil$\crcr\noalign{\kern1.5pt\nointerlineskip}}}%
		\vphantom{\copy1}%
		\setbox2=\hbox{$\textstyle \sim$}%
		\wd1=\wd0\dp1=0cm\ifdim\wd2>\wd1 \wd1=\wd2\else\relax\fi
		\ht1=\ht0\relax
		\box1}{\setbox0=\hbox{$\scriptstyle #1$}
		\setbox1=\vtop{\ialign{##\crcr
				$\hfil{\scriptstyle #1}\hfil$\crcr\noalign{\kern1pt%
					\nointerlineskip}$\hfil\mathord{\scriptstyle \sim}%
				\hfil$\crcr\noalign{\kern2.1pt\nointerlineskip}}}%
		\setbox2=\hbox{$\scriptstyle \sim$}%
		\vphantom{\copy1}%
		\wd1=\wd0\dp1=0cm\ifdim\wd2>\wd1 \wd1=\wd2\else\relax\fi
		\ht1=\ht0\relax
		\box1}{\setbox0=\hbox{$\scriptscriptstyle #1$}
		\vtop{\ialign{##\crcr
				$\hfil{\scriptscriptstyle #1}\hfil$\crcr\noalign{\kern1pt%
					\nointerlineskip}$\hfil\mathord{\scriptscriptstyle \sim}%
				\hfil$\crcr\noalign{\kern1.5pt\nointerlineskip}}}%
	}}

\newtheorem{lemma}[ttt]{Lemma}
\newtheorem{theorem}[ttt]{Theorem}

\newtheorem{cl}[ttt]{Claim}

\newtheorem{question}[ttt]{Question}
\newtheorem{rem}[ttt]{Remark}
\newtheorem{cor}[ttt]{Corollary}

\begin{document}

\title{On the spectra of cardinalities of branches of Kurepa trees}

\author{Márk Poór}
\address{E\"otv\"os Lor\'and University, Institute of Mathematics, P\'azm\'any P\'eter
	s. 1/c, 1117 Budapest, Hungary}
\email{sokmark@caesar.elte.hu}
\thanks{The author was supported by the National
	Research, Development and Innovation Office – NKFIH, grant no. 104178, 124749 and 129211.\\
\raisebox{0.5cm}{\parbox[c]{11cm}{\scshape Supported by the \'UNKP-18-3 New National Excellence Program of the Ministry of Human Capacities.}}}

\maketitle

\begin{abstract}
We are interested in the possible sets of cardinalities of branches of Kurepa trees in models of $ZFC$ $+$ $CH$.
In this paper we present a sufficient condition (for sets of cardinals) to be consistently the set of cardinalities of  branches of Kurepa trees.
\end{abstract}

\section{Introduction}
Kurepa trees are trees of height $\omega_1$ which have countable levels, but have more than $\omega_1$-many cofinal branches.
We are interested in how those sets of cardinals look like for which there is a model of $ZFC$ such that the following holds:
\[ \text{a cardinal } \kappa \text{ is an element of our fixed set } S \]
\[ \text{ iff } \]
\[ \text{ there is a Kurepa tree } T \text{ with exactly }\kappa \text{-many cofinal branches.} \]
J.H. Silver showed that the existence of Kurepa trees is independent of $ZFC + CH$ \cite{silb}. R. Jin. and S. Shelah constructed a model of $CH$ and $2^{\omega_1} = \omega_4$, where Kurepa trees only with $\omega_3$-many cofinal branches exist, and another with $2^{\omega_1} > \omega_2$ and Kurepa trees having exactly $2^{\omega_1}$-many cofinal branches \cite{Jin}. Moreover, in the latter there are no Jech-Kunen trees (a tree of height $\omega_1$ is a Jech-Kunen tree, iff each level is of power at most $\omega_1$, with the cardinality of cofinal branches strictly between $\omega_1$ and $2^{\omega_1}$).  In \cite{Jin2} they prove that it is consistent with $ZFC + CH$ that there are Jech-Kunen trees, but no Kurepa trees.

If for a given sequence of cardinals $\langle \kappa_i : \ i \in \omega\rangle$ we have
 Kurepa trees $T_i$ ($i \in \omega$) with each $T_i$ having $\kappa_i$-many cofinal branches, then $\cup \{T_i: \ i \in \omega \}$ is a Kurepa tree with $\sup \{ \kappa_i: \ i \in \omega \}$-many cofinal branches. This means that the set of cardinalities of Kurepa trees is closed under taking limits of countable sequences.
Similarly, it is not hard to see that this set is closed under taking limits of $\omega_1$-long sequences too.
Our goal is that for a given set $S$ of cardinals satisfying a slightly strengthened version of this necessary condition (in a countable transitive model [c.t.m.] $M$ of $ZFC$) to construct a forcing extension $M' \supseteq M$ where
\[  M' \models S = \{ \kappa: \ \text{ there exists a Kurepa tree } T \text{ such that } |\mathcal{B}(T)| = \kappa \}  \]
(where $\mathcal{B}(T)$ denotes the set of cofinal branches of $T$).
Our main result, Theorem  $\ref{foo}$ implies that, for example, consistently there are Kurepa trees with  $\aleph_{2k}$-many ($k \in \omega$), $\aleph_\omega$- and $\aleph_{\omega+1}$-many cofinal branches (but there are no Kurepa trees with $\aleph_{2k+1}$-many cofinal branches).

\section{Preliminaries and notations}

In this paper, all ordinals are von Neumann ordinals, and by the cardinality of a set $S$ (in symbols $|S|$) we mean the least ordinal $\alpha$ such that there exists a bijection between $\alpha$ and $S$.
For any function $f$ with domain $\dom(f) = S$, 
the following sequencelike notation will also symbol the set $f$
\[ \langle f_s : \ s \in S \rangle, \]
that is
\[ f = \{ \langle s, f(s) \rangle : \ s \in S \} =  \langle f_s : \ s \in S \rangle. \]
For a given set $S$, and ordinal $\beta$, $S^\beta $ will symbol the set of functions from $\beta$ to $S$, i.e.  $S^\beta = \{ f: \beta \to S \}$. 
Similarly $S^{<\beta} = \bigcup_{\alpha < \beta} S^\alpha $.
We use the notation (for any set $S$ and cardinal $\lambda$)  $[S]^{\lambda} = \{ H \in \mathcal{P}(S): \ |H| = \lambda \}$,
and similarly $[S]^{<\lambda} = \{ H \in \mathcal{P}(S): \ |H| < \lambda \}$.
Regarding forcing we refer to \cite{tools} and \cite{kunen}.
\bd A tree $\langle T, \prec_T \rangle$ is a partially ordered set (poset) in which for each $x \in T$
	the set
	\[ T_{\prec x} = \{ y \in T: \ y \prec_T x \} \]
is well ordered by $\prec_T$.
\ed

\bd
	The height of $x$ in the tree $T$  is the  order type of $T_{\prec x}$
	\[ \htt(x,T) =  \otp(T_{\prec x}). \]
\ed

\bd
	For each ordinal $\alpha$ the $\alpha$'th level of $T$, or $\Le_\alpha(T)$ is 
	\[ \{ x \in T: \ \htt(x,T) = \alpha \}.  \]
	The restriction of $T$ to $\alpha$  is 
	\[ T\restrict{\alpha} = \cup \{ \Le_\beta(T) : \ \beta < \alpha \} .\]
\ed

\bd
	The height of the tree $T$, or $\htt(T)$  is the least $\beta$ such that
	\[ \Le_\beta(T) = \emptyset. \]
	A tree of height $\omega_1$, with $|\Le_\alpha(T)| < \omega_1$ for each $\alpha$ is called an $\omega_1$-tree.
\ed

From now on by branch we will mean cofinal branch.
\bd
	A branch of a tree $T$ is an ordered set (w.r.t. $\prec_T$) containing exactly one
	 element of each $\Le_\alpha(T)$, $\alpha < \htt(T)$. For a given tree $T$ $\mathcal{B}(T)$ denotes the set of branches of $T$.
\ed

\bd
	An $\omega_1$-tree $T$  is a Kurepa tree if
	it has more than $\omega_1$-many branches.
\ed

In this paper we will restrict our attention to trees 
that are downward closed subsets of $ 2^{<\omega_1}$,
i.e. a set $T$ of $0-1$-valued functions on countable ordinals, where 
\begin{equation} \label{fa}
(T \subseteq 2^{<\omega_1})  \ \text{ and } \ f \in T, \ \beta < \dom(f)  \ \text{ implies } \ f\restrict{\beta} \in T ,\end{equation}
and $f \prec_{T} g$, iff $g$ extends $f$ as a function ($f \subseteq g$). 
Then it is easy to see that for each $f \in T$, $\htt(f,T) = \dom(f)$,
and $\Le_\gamma(T) = T \cap 2^\gamma$.)
The following well-known lemma states that regarding our problem,  we can assume that trees are of the form as in $\eqref{fa}$.

\begin{lemma}
	Suppose that $T$ is an $\omega_1$-tree. Then there exists an $\omega_1$-tree $T' \subseteq 2^{<\omega_1}$ which is downward closed (i.e. $T'$ is of the form $\eqref{fa}$) with the same cardinality of branches, i.e.
	\[ |\mathcal{B}(T)|= |\mathcal{B}(T')|  \] 
\end{lemma}
\begin{proof}
	First we need the following claim.
\begin{cl}
	Assume that $\langle T, \prec_T \rangle $ is a tree such that $|\Le_\alpha(T)| \leq 2^\omega$, $\Le_{\omega_1} = \emptyset$. Then there is an order-preserving mapping $F: T \to 2^{<\omega_1}$ such that
	for the downward closed tree generated by $\ran(F)$  
	\[ T' = \{f \in 2^{<\omega_1}: \ f \subseteq F(t) \text{ for some } t \in T \} \subseteq 2^{<\omega_1} \]
	the following holds. For each $\alpha <\omega_1$ for the $\omega \cdot (\alpha+1)$-th level of $T'$ 
	\begin{equation} \label{szint} \Le_{\omega \cdot (\alpha+1)}(T') =2^{\omega \cdot (\alpha+1)} \cap T' = \{ F(t): \ t \in \Le_\alpha(T) \}. \end{equation} 
\end{cl}
\begin{proof}
	Fix an injection $G: T \to 2^\omega$. For each $t \in T$ if $t \in \Le_\alpha(T)$ then define
	for each $\beta < \alpha$ the element $t_\beta$ to be the unique element in $\Le_\beta(T)$ under $t$, i.e., 
	\[ t_\beta \prec_T t, \ \ t_\beta \in L_\beta(T), \] 
	and let
	\[ t_\alpha = t. \]
	 Now define $F(t) \in 2^{\omega \cdot (\alpha+1)}$ as follows.
	\begin{equation}\label{Fdf} (F(t))(\beta \cdot \omega + n) = (G(t_{\beta}))(n) \ (\beta \leq \alpha, \ n \in \omega). \end{equation}
	Now if $f \in 2^{\omega \cdot (\gamma+1)}$ ($\gamma <\omega_1$) is the restriction of $F(t)$ for some $t \in T$, i.e. $f = F(t)\restrict{\omega \cdot (\gamma+1)}$, and
	if $\alpha$ is such that $t \in \Le_\alpha(T)$, then clearly $\alpha \geq \gamma$. Using that $t_\beta = (t_\gamma)_\beta$ if $\beta \leq \gamma$,
	 we obtain by $\eqref{Fdf}$ that
	\[ f(\beta \cdot \omega + n) = F(t)\restrict{\omega \cdot (\gamma+1)}(\beta \cdot \omega + n) = (G(t_{\beta}))(n) = (G((t_{\gamma})_\beta))(n)  \ (\beta \leq \gamma, \ n \in \omega),\]
	therefore $f = F(t_\gamma)$.

\end{proof}
Obviously $|\mathcal{B}(T)| \leq |\mathcal{B}(T')|$.

The function $F: T \to T'$ given by the claim is an order-preserving embedding from $\langle T, \prec_T \rangle$ to $\langle T', \subseteq \rangle$. 
Moreover, 
the fact that the $\omega \cdot (\alpha +1)$-th level of $T'$ is the $F$-image of $\Le_\alpha(T)$ (by $\eqref{szint}$) implies that for each cofinal branch $b \subseteq T'$ and $\alpha < \omega_1$
\[ b \cap 2^{\omega \cdot (\alpha+1)} = b \cap \Le_{\omega \cdot (\alpha+1)}(T') = \{ F(t) \} \ \ \text{ for some } t \in \Le_\alpha(T). \]
This means that (fixing $b \in \mathcal{B}(T')$) by the order-preservation
\[ \{ t \in T: \ F(t) \in b \} \ \text{ is a branch in } T,\]
therefore $|\mathcal{B}(T)| \geq |\mathcal{B}(T')|$, indeed.

\end{proof}

Before stating our main theorem, we need some technical preparations.

\bd
	Let the ordinal $\alpha \leq \omega_1$ be given, and let $a,b \in 2^\alpha$. We define the mapping 
	$F_{ab} : 2^{\leq \omega_1} \to 2^{\leq \omega_1}$ as follows
	\[ \begin{array}{ccc}
	 
	 \end{array}
	 \]
	 \[ \begin{array}{ll}
	 s \in 2^\beta  \mapsto & F_{ab}(s) \in 2^\beta  \\
	 & \beta \ni \gamma \mapsto \left\{  \begin{array}{lr}  s(\gamma) + a(\gamma) + b(\gamma)  \ \ (\mo \ 2) & \text{ if } \ \gamma < \alpha \\ s(\gamma)  & \ \ \gamma \geq \alpha \end{array}  \right. \end{array} \]
	 
\ed
	It can be easily seen that $F_{ab}$ is an automorphism of the tree $\langle 2^{<\omega_1}, \subseteq \rangle$.

\bd
	A tree $T \subseteq 2^{<\omega_1}$ which is downward closed is said to be homogeneous if for each pair $a,b \in \Le_\alpha(T)$ on the same level,
	$F_{ab}$ is an automorphism of $\langle T, \subseteq \rangle$.
\ed

\bd \label{normaltree}
	A tree $\langle T, \prec_T \rangle$ is normal if the following conditions hold
	\begin{itemize}
		\item each $t \in T$ which is not on the top level of $T$ has at least two immediate successors in $T$,
		\item for each $t \in \Le_\alpha(T)$, and each $\beta > \alpha$ (where $\beta < \htt(T)$) there exists an element $t' \in \Le_\beta(T)$, $t \prec_T t'$,
		\item for each limit $\alpha $ (where $\alpha < \htt(T)$) and  $b \in \mathcal{B}(T\restrict{\alpha})$, there is at most one common upper bound of $b$ in $\Le_\alpha(T)$.
	\end{itemize}
\ed

\bd \label{PhomD}
	The set
\[ \mathbb{P}_{\hom} = \{ T \subseteq 2^{<\omega_1}: \ T \text{ is a countable homogenous normal subtree } \}  \]
is a notion of forcing with the partial order
\[ T \leq T' \ \ \iff  \ T\restrict{\htt(T')} = T'  ,\]
i.e. the condition  $T$ extends the condition $T'$ iff the tree $T$ is an end-extension of $T'$.

\ed
(It is easy to see that a $\mathbb{P}_{\hom}$-generic filter corresponds to a homogeneous subtree of $2^{<\omega_1}$ of height $\omega_1$.)

 \bd
 	A partial order $\mathbb{P}$ is $\lambda$-closed, if whenever  $\langle p_\alpha: \ \alpha<\gamma \rangle$
 	is a decreasing sequence (i.e. $\beta < \alpha$ implies $p_\beta \geq p_\alpha$ ) of length $\gamma<\lambda$, then  
 	there exists a common lower bound $p \in \mathbb{P}$, i.e. $p \leq p_\alpha$ for each $\alpha < \gamma$.
 
 \ed

\begin{lemma} \label{omclos}
	$\mathbb{P}_{\hom}$ is $\omega_1$-closed.
\end{lemma}
\begin{proof}
	If a decreasing sequence $p_0 \geq p_1 \geq \dots $ is given, then $\cup \{p_i: \ i \in \omega \}$ is a 
	growing union of countable homogeneous normal trees.
	Since it is easy to check that the growing union of normal trees is normal,
	 and homogeneity of a tree $T$ means that
	for $a,b,t \in T$ ($a,b$ are on the same level) $F_{ab}(t) \in T$, we are done.
\end{proof}

At some point we will make use of the following claim.
\begin{cl} \label{segit}
Let $T\subseteq 2^{<\omega_1} $ be a homogeneous tree, $t, t',t'' \in T$, 
and $\alpha = \htt(t) =\htt(t' )< \htt(t'') = \beta$, that is $t$ and $t'$ are on the same level, and $t''$ is on a higher level.
Furthermore, assume that $t' \sq t''$, i.e. $t''$ is an extension of $t'$ as a function.
 Then $t \cup t''\restrict{\beta \setminus \alpha} \in T$, that is,
roughly speaking, $t$ and $t'$ have the same extensions in $T$.
\end{cl}
\begin{proof}
	It is easy to check that $F_{tt'}(t'') = t \cup t''\restrict{\beta \se \alpha}$.
\end{proof}

	

We will make use of the next lemma which is \cite[VII., Thm. 6.14.]{kunen}
\begin{lemma} \label{ujfv}
	Let $M$ be a c.t.m. Suppose that the cardinal $\lambda$, and the sets $A,B \in M$ ($|A| < \lambda$) are given. Let $\mathbb{P}$ be a $\lambda$-closed notion of forcing, $G \subseteq \mathbb{P}$ be  $\mathbb{P}$-generic over $M$, $f: A \to B$, $f \in M[G]$. Then $f \in M$.
\end{lemma}
Which has the following straightforward corollary.
\begin{cor}
	If $\beta < \lambda$, then forcing with a $\lambda$-closed notion of forcing adds no new subsets of $\beta$.
\end{cor}

The next lemma is a corollary of the proof of Lemma $\ref{ujfv}$. It is folklore. 
(Recall that for each element $x \in M$ there is a canonical $\mathbb{P}$-name 
\begin{equation} \label{checck} \widehat{x} = \{ \widehat{y} \times \{ \mathbb{1}_\mathbb{P} \}  : \ y \in x \} \end{equation}
for which the evaluation  of  $\widehat{x}$ by $G$
\[ \widehat{x}[G] = x \]
 whenever $G \sq \mathbb{P}$ is $\mathbb{P}$-generic over $M$, see \cite[Ch VII, Definition 2.10]{kunen}.) 
\begin{lemma} \label{ujfvneve}
		Let $M$ be a c.t.m. Suppose that the cardinal $\lambda$, and the sets $A,B \in M$ ($|A| < \lambda$) are given. Let $\mathbb{P}$ be a $\lambda$-closed notion of forcing, $p \in \mathbb{P}$, $f$ is a $\mathbb{P}$-name for which
		 \[ p \Vdash f: A \to B \text{ is a function.} \] 
		 Then there is an extension $p' \leq p$, and a function $f_0 \in M$ such that
		 \[ p' \Vdash f = \widehat{f_0}.  \]
\end{lemma}
\begin{proof}
	Let $G \subseteq \mathbb{P}$-generic over $M$ with $p \in G$. Then apply Lemma $\ref{ujfv}$, set $f_0 = f[G] \in M$, and choose 
	$p'\in G$, $p' \leq p$ such that $p' \Vdash f = \widehat{f_0}$.
\end{proof}

The lemma has the following straightforward application.
\begin{cor} \label{nincsuj}
	Forcing with a $\lambda$-closed notion of forcing  adds no new sequences of type $\gamma$ (for any $\gamma < \lambda$), that is, if $G$ denotes the generic filter, then
	\[ M^\gamma \cap M[G] = M^\gamma \cap M. \]
\end{cor}

For some technical reasons we will later
use the following definition and lemma.
\bd
	If $M$ is a c.t.m., $\mathbb{P} \in M$ is a notion of forcing and $\sigma, \tau \in M$ are $\mathbb{P}$-names, then $\sigma$ is a nice $\mathbb{P}$-name for a subset of $\tau$ if $\sigma$ is of the form
	\[ \sigma = \cup \{ \{ \pi \} \times A_\pi : \ \pi \in \dom(\tau), \ A_\pi \subseteq \mathbb{P} \text{ is an antichain} \}.  \]
	
\ed
The next lemma is  \cite[Ch VII, Lemma 5.12]{kunen}.
\begin{lemma} \label{nice}
	Suppose that $M$ is a c.t.m., $\mathbb{P} \in M$ is a notion of forcing,  $\tau, \mu \in M$ are $\mathbb{P}$-names. Then there is a nice name $\sigma$ for a subset of $\tau$ such that
	\[ 1 \Vdash  (\mu \subseteq \tau) \rightarrow (\mu = \sigma) . \]
\end{lemma}

The following lemma is folklore, but for the sake of completeness we include the proof.
\begin{lemma} \label{hatvanyh}
	Let $M$ be a c.t.m., $\lambda$, $\varrho$ be cardinals in $M$,  $\mathbb{P} \in M$ be a notion of forcing which is $\lambda$-cc.
	Then, whenever $G \subseteq \mathbb{P}$ is generic over $M$, and $\nu$ is such that
	\[ M \models \ \nu = (|\mathbb{P}|^{<\lambda})^\varrho, \]
	then
	\[ M[G] \models \ 2^\varrho \leq \nu. \]
\end{lemma}
\begin{proof}
	By $\lambda$-cc, there are at most $|\mathbb{P}|^{<\lambda}$-many antichains in $\mathbb{P}$. This implies that since $\dom(\widehat{\varrho})= \{ \widehat{\alpha} \ : \alpha < \varrho \}$, there are at most $(|\mathbb{P}|^{<\lambda})^\varrho = \nu$-many nice names for subsets of $\widehat{\varrho}$ in $M$. Let $C \in M$ denote the set of nice names for subsets of $\widehat{\varrho}$ (where $|C| \leq \nu$).
	By Lemma $\ref{nice}$ each subset $S \in \mathcal{P}(\varrho) \cap M[G]$ is represented by a nice name $\tau \in C$.
\end{proof}

The following lemma can be found as \cite[Ch. VII., Lemma 6.9]{kunen}
\begin{lemma} \label{cc-cf}
	Let $\lambda$ be a cardinal in a c.t.m. $M$, and $\mathbb{P}$ be a poset which is $\lambda$-cc in $M$. Then forcing with $\mathbb{P}$
	preserves cofinalities $\geq \lambda$, i.e. if 
	\[ \cf^M(\alpha) \geq \lambda, \]
	then whenever $G$ is $\mathbb{P}$-generic over $M$,
	\[ \cf^M(\alpha) = \cf^{M[G]}(\alpha), \]
	in particular if $\lambda$ is regular in $M$, then $\mathbb{P}$ preserves cardinals $\geq \lambda$.
\end{lemma}
The following well-known fact can be found as \cite[Ch. VIII. Lemma 3.4]{kunen}.
\begin{lemma} \label{nemadujagat} Suppose  that $T \in M$ is an $\omega_1$-tree, $\mathbb{P}$ is $\omega_1$-closed. Then forcing with $\mathbb{P}$ does not add any new branch to $T$.
\end{lemma}

The next technical lemma will be later needed \cite[Ch. VII. Lemma 7.11]{kunen}.
\begin{lemma} \label{surube}
	Let the poset $\mathbb{P} \in M$ given, and suppose that $\mathbb{P}' \in M$ is a dense subset of $\mathbb{P}$. Then
	\begin{itemize}
		\item if $G \subseteq \mathbb{P}$ is generic over $M$, then the intersection $G' = G \cap \mathbb{P}'$ is $\mathbb{P}'$-generic over $M$,
		\[ M[G] = M[G'],\]
		moreover,  $G = \{ p \in \mathbb{P}: \ \exists q \in G' \ q \leq p \}$,
		\item if $G' \subseteq \mathbb{P}'$ is generic over $M$, then  the filter $G = \{ p \in \mathbb{P}: \ \exists q \in G' \ q \leq p \}$ is $\mathbb{P}$-generic over $M$, 
			\[ M[G'] = M[G],\]
			moreover,  $G' = G \cap \mathbb{P}'$.
	\end{itemize}
\end{lemma}

\section{The main result}

In this section we will prove the following theorem. We will make use of some ideas from \cite{Jin}, where the authors proved among others, that it is consistent with $ZFC$ that $2^{\omega_1} = \omega_4$, and Kurepa trees only with $\omega_3$-many branches exist.

Our main forcing object will be a two-step forcing iteration where we can
isolate a dense closed subset. The first paper in which such argument arose was \cite{satur} (see the end of Section 3 of
that paper. In fact, the forcings used by Kunen are essentially versions of the forcings used in this paper.)

\bd \label{closedunder}
	A set of ordinals $E$ is closed under taking $\varrho$-limits, iff whenever $\delta$ is an ordinal such that $\cf(\delta) = \varrho$ and $E \cap \delta$ is cofinal in $\delta$, then $\delta \in E$.
\ed

\begin{theorem} \label{foo}
	Let $M$ be a c.t.m. of $ZFC + GCH$, and let $C \in M$ be a set of ordinals such that $0,1 \notin C$, and the following holds (in $M$).
	\begin{equation}  \label{cofos} \begin{array}{c} C \text{ is closed under }\omega\text{-, and }\omega_1\text{-limits, and} \\
	\forall \delta \in C:  \ \ \omega \leq \cf(\delta) \leq \omega_1 \ \text{ implies } \delta+1 \in C, 

	\end{array} \end{equation}
	If $2 \notin C$, then further assume that 
	\begin{equation} \label{limzart}
	\begin{array}{c}
	\text{ there is an inaccessible cardinal }\kappa \in M, \text{moreover} \\
				C \text{ is closed under }<\kappa \text{-limits, and} \\
				\forall \delta \in C: \ \ \omega \leq \cf(\delta) <\kappa \ \text{ implies } \delta+1 \in C,
				\end{array} 
				\end{equation}
	Then there is notion of forcing $\mathbb{S} \in M$ such that whenever $G \subseteq \mathbb{S}$ is $\mathbb{S}$-generic over $M$, then
	\[ M[G] \models  C = \{ \alpha: \ \text{ there exists a Kurepa tree } T \text{ such that } |\mathcal{B}(T)| = \omega_\alpha \} \]
\end{theorem}
\begin{rem}
	If $2 \notin C$, then in the final model $\kappa$ will be $\omega_2$, thus $\eqref{limzart}$ is requiring condition $\eqref{cofos}$ to be true in \emph{the final model}.
\end{rem}

First we define $\mathbb{S}$. We will work in $M$.

If $2 \notin C$, i.e. we would like to obtain a final model in which every Kurepa tree has more than $\omega_2$ branches,
then define $\mathbb{L}$ to be the following Lévy collapse
\begin{equation} \label{Levy} \begin{array}{l} \mathbb{L} = Lv(\kappa,\omega_1) = \\
= \{ f: \ \dom(f) \subseteq \kappa \times \omega_1, \ \dom(f)\in [\lambda \times \kappa]^{<\omega_1}, \ f(\lambda, \alpha) < \lambda \ \ (\forall \ \langle \lambda, \alpha \rangle \in  \kappa \times \omega_1)  \}. \end{array} \end{equation}

Let 
\beeq \label{Pdef} \mathbb{P}_\alpha = \mathbb{P}_{\hom} \text{ for each }\alpha \in C, \eeq 
and let $\name{T_\alpha}$ be the $\mathbb{P}_\alpha$-name of 
the generic tree. 

We have two distinct cases depending on whether $2 \in C$.
We will need the following sets defined for each ordinal in $C$.
\bd \label{Xhalmazok}
	Let the system
	\[ \langle X_\alpha: \ \alpha \in C\rangle \ \ (X_\alpha \cap X_\beta = \es,  \ \alpha \neq \beta \in C)\] 
	of pairwise disjoint sets such that
\begin{itemize}
	\item If $2 \in C$, then
	\[ |X_\alpha|= \omega_\alpha, \]

	\item otherwise, if $2 \notin C$ then 
	\[  |X_\alpha| = \left\{ \begin{array}{ll} \omega_{\kappa + \alpha-2},  & \text{ if } \alpha < \omega, \\ 
							\omega_{\kappa + \alpha}, &  \text{ if } \alpha \geq \omega. \end{array} \right. \]
	Observe that if one collapses each cardinal greater than $\omega_1$ and less than $\omega_\kappa = \kappa$ (where each other cardinal remains a cardinal), then in that model $|X_\alpha| = \omega_\alpha$.
\end{itemize}
\ed

\bd \label{Qdef}
Let $\name{\mathbb{Q}_{\alpha}}$, $\name{\mathbb{1}_{\mathbb{Q}_{\alpha}}}$ , $\name{\leq_{\mathbb{Q}_{\alpha}}}$ be $\mathbb{P}_{\alpha}$-names in $M$ for which
\begin{equation} \label{ite1}  \mathbb{1}_{\mathbb{P}_{\alpha}} \Vdash \ \name{\mathbb{Q}_{\alpha}} =  \{ f:  \ \dom(f)\in  [\widehat{X_{\alpha}}]^{<\omega_1}, \ \ran{f}\subseteq \name{T_\alpha}  \},   \end{equation}
with the pointwise extension order, i.e.
\begin{equation} \label{ite2} \mathbb{1}_{\mathbb{P}_{\alpha}} \Vdash \ \name{\leq_{\mathbb{Q}_{\alpha}}} = \{ \langle f,g \rangle : \ \forall x \in \dom(g): \ (x \in \dom(f) \wedge f(x) \supseteq g(x)) \}, \end{equation}
and a name for the greatest element $\name{\mathbb{1}_{\mathbb{Q}_{\alpha}}}$
\begin{equation} \label{ite3}  \mathbb{1}_{\mathbb{P}_{\alpha}} \Vdash \ \name{\mathbb{1}_{\mathbb{Q}_{\alpha}}} \in \name{\mathbb{Q}_{\alpha}} \text{ is the empty function, i.e. the empty set.} \end{equation} 
\ed

Such names exist by the maximal principle \cite[II., Thm. 8.2]{kunen}.
Now after one adds a $\mathbb{P}_{\alpha}$-generic filter $F$ over $M$, $\name{\mathbb{Q}_{\alpha}}$ will be decoded into a partial order with the reverse inclusion relation, with the largest element $\name{\mathbb{1}_{\mathbb{Q}_{\alpha}}}[F]$.
\br  \label{Q1}
After replacing $\name{\mathbb{Q}_{\alpha}}$ by $\name{\mathbb{Q}_{\alpha}} \cup \{ \langle \emptyset, \mathbb{1}_{\mathbb{P}_{\alpha}} \rangle \}$ (if needed) we can assume that
\begin{equation} \label{1Q}
 \name{\mathbb{1}_{\mathbb{Q}_{\alpha}}} = \emptyset = \widehat{\emptyset},
\end{equation}
and
\begin{equation} \label{1Qv}
 \langle \name{ \mathbb{1}_{\mathbb{Q}_{\alpha}}},  \mathbb{1}_{\mathbb{P}_{\alpha}} \rangle =  \langle \emptyset, \mathbb{1}_{\mathbb{P}_{\alpha}} \rangle \in \name{\mathbb{Q}_{\alpha}}
\end{equation}
\er

Now we define $\mathbb{R}_{\alpha}$-s ($\alpha \in C$) to be the following two step iterations as in \cite[Ch VIII., §5.]{kunen}.
\bd \label{Rde}
\begin{equation} \label{iter} \mathbb{R}_{\alpha}= \mathbb{P}_{\alpha} \ast \name{\mathbb{Q}_{\alpha}} = \{ \langle p , \name{q} \rangle : \ p \in \mathbb{P}_{\alpha}, \name{q} \in \dom(\name{\mathbb{Q}_{\alpha}}), \ p \Vdash_{\mathbb{P}_{\alpha}} \name{q} \in \name{\mathbb{Q}_{\alpha}} \} ,\end{equation}
\ed
which is a notion of forcing with the following partial order
\[ \langle p_1,\name{q_1} \rangle \leq \langle p_2,\name{q_2} \rangle  \ \iff \ p_1 \leq p_2 \ \wedge \ p_1 \Vdash_{\mathbb{P}_{\alpha}} \name{q_1} \leq \name{q_2}, \]
and a (not necessarily unique) greatest element
\begin{equation} \label{1R} \mathbb{1}_{\mathbb{R}_{\alpha}} = \langle \mathbb{1}_{\mathbb{P}_{\alpha}}, \name{\mathbb{1}_{\mathbb{Q}_{\alpha}}} \rangle  \end{equation}

\bd \label{R^E}
For our fixed set $C$ let $\mathbb{R}$ be the following countably supported product 
\begin{equation} \label{Rdef} \mathbb{R} = \{ r \in \prod_{\alpha \in C} \mathbb{R}_{\alpha}: |\supp(r)|<\omega_1 \},
\end{equation}
(where by $\supp(r)$ we mean the set $\{ \alpha \in C: \ r_\alpha \neq \mathbb{1}_{\mathbb{R}_{\alpha}} \}$)
which is a partial order with the product ordering, i.e.
\[ r_0 \leq r_1 \ \iff \  \forall \alpha \in C \ \ (r_0)_\alpha \leq (r_1)_\alpha .\]
For a set $E \subseteq  C$ define
\beeq  \label{RC}
\mathbb{R}\restrict{E} = \{ r\restrict{E}: \ r \in \mathbb{R} \}.
\eeq
\ed
Clearly for any partition $E_1$, $E_2$ of $C$, 
\[ \mathbb{R} \simeq \mathbb{R}\restrict{E_1} \times  \mathbb{R}\restrict{E_2}. \]


Now we can define $\mathbb{S}$.
\bd
\begin{equation} \label{Rdf}
 \mathbb{S} = \left\{ \begin{array}{ll} \mathbb{L} \times \mathbb{R}, & \text{ if } 2 \notin C, \\
	   \mathbb{R} & \text{ if } 2 \in C.
  \end{array} \right.
 \end{equation}
 (Where by $\mathbb{L} \times \mathbb{R}$ we mean the product the partial order, i.e. $\langle l_1,r_1 \rangle \leq \langle l_2,r_2\rangle$ iff $l_1 \leq l_2$ and $r_1 \leq r_2$.)
\ed
If $r \in \mathbb{R}$  and $\alpha \in C$, then $\pr_\alpha(r) \in \mathbb{R}_{\alpha}$ denotes its projection onto its $\alpha$-th coordinate.

From now on we fix an $\mathbb{S}$-generic filter $G$ over $M$. For any 
$\alpha \in C$ set $G_{\alpha}$ to be $G$-s projection onto $\mathbb{R}_{\alpha}$.
Similarly, for any set $E \subseteq C$
let $G\restrict{E} \subseteq \mathbb{R}\restrict{E}$ denote $G$-s projection onto $\mathbb{R}\restrict{E}$. The following
lemma \cite[Ch VIII., Lemma 1.3]{kunen} guarantees that $G\restrict{E}$ is $\mathbb{R}\restrict{E}$-generic over $M$.
\begin{lemma} \label{szorz1}
	Let $\mathbb{P}_0 \times \mathbb{P}_1$ be a product partial order, and fix a filter $G$ which is $\mathbb{P}_0 \times \mathbb{P}_1$-generic over $M$. Then $\pr_0(G) \subseteq \mathbb{P}_0$ is $\mathbb{P}_0$-generic over $M$, $\pr_1(G) \subseteq \mathbb{P}_1$ is $\mathbb{P}_1$-generic over $M$, and $G = \pr_0(G) \times \pr_1(G)$.
\end{lemma}
We will make use of the following too \cite[Ch VIII., Thm. 1.4]{kunen}.
\begin{lemma} \label{szorz2}
	Let $G_0 \subseteq \mathbb{P}_0$, $G_1 \subseteq \mathbb{P}_1$ be filters, then the following three conditions are equivalent.
	\begin{enumerate}[(1)]
		\item \label{el} $G_0 \times G_1$ is $\mathbb{P}_0 \times \mathbb{P}_1$-generic over $M$,
		\item $G_0$ is $\mathbb{P}_0$-generic over $M$, and $G_1$ is $\mathbb{P}_1$-generic over $M[G_0]$,
		\item  \label{ut} $G_1$ is $\mathbb{P}_1$-generic over $M$, and $G_0$ is $\mathbb{P}_0$-generic over $M[G_1]$.
	\end{enumerate}
    Furthermore if $\eqref{el}-\eqref{ut}$ holds, then 
    \[ M[G_0 \times G_1]  = M[G_0][G_1] = M[G_1][G_0] .\]
	
\end{lemma}


The next definition, and lemma can help us to find an intermediate model between $M$, and $M[G_{\alpha}]$ (for a fixed ordinal $\alpha \in C$).
\bd
	Let $\mathbb{P}$ be a partial order in $M$, and let $\name{\mathbb{Q}}$ be a $\mathbb{P}$-name for a partial order. If the filter $F \subseteq \mathbb{P}$ is $\mathbb{P}$-generic over $M$, and $H \subseteq \name{\mathbb{Q}}[F] \in M[F]$, then 
	\[ F \ast H = \{ \langle p, \name{q} \rangle \in \mathbb{P} \ast \name{\mathbb{Q}} : \ p \in F, \name{q}[F] \in H \}.  \]  
\ed

We state \cite[Ch VIII. Thm. 5.5]{kunen}
\begin{lemma} \label{altiterlem}
	Let $\mathbb{P}$ be a partial order in $M$, and let $\name{\mathbb{Q}}$ be a $\mathbb{P}$-name for a partial order.  Let $G \subseteq \mathbb{P} \ast \name{\mathbb{Q}}$ be a filter, $F = \pr_\mathbb{P}(G)$, and let
	\[ H = \{ \name{q}[F] : \  \exists p: \ \langle p,q \rangle \in G \} .  \]
	If $G$  is $\mathbb{P} \ast \name{\mathbb{Q}}$-generic over $M$, then 
	\begin{itemize}
		\item $F$ is $\mathbb{P}$-generic over $M$,
		\item $H \subseteq \name{\mathbb{Q}}[F] \in M[F]$ is $\name{\mathbb{Q}}[F]$-generic over $M[F]$,
		\item $G= F \ast H$, and
		\item $M[G] = M[F][H]$.
	\end{itemize}
	 
\end{lemma}
Using this, and having the filter $G_{\alpha}$ which is $\mathbb{R}_{\alpha}$-generic over $M$,
  we can define 
  \begin{equation} \label{FH}
  \begin{array}{rl}
   F_{\alpha} =  & \pr_{\mathbb{P}_{\alpha}}(G_{\alpha}) \ (\subseteq \mathbb{P}_{\alpha}), \\ 
 H_{\alpha}  = & \{ \name{q}[F_{\alpha}] : \  \exists p \ \langle p,q \rangle \in G_{\alpha} \} \subseteq \name{\mathbb{Q}_{\alpha}}[{F_{\alpha}}], \end{array} \end{equation}
(where $F_{\alpha}$ is $\mathbb{P}_{\alpha}$-generic over $M$) so that
\begin{equation} \label{GFH} G_{\alpha} = F_{\alpha} \ast H_{\alpha}, \end{equation}
and
\begin{equation} \label{MGFH}
M [G_{\alpha}] = M[F_{\alpha}][H_{\alpha}]
\end{equation}

holds.

Now we will verify some technical statements about the aforementioned partial orders.
\bd
For each $\alpha \in C$ we define the subset $\mathbb{R}^\bullet_\alpha \subseteq \mathbb{R}_\alpha$ as follows. Let $\langle p, \name{q} \rangle \in \mathbb{R}^\bullet_\alpha$, iff ($\langle p, \name{q} \rangle \in \mathbb{R}_\alpha$, and)
\begin{enumerate}
 	\item	$p \in \mathbb{P}_{\alpha} = \mathbb{P}_{\hom}$, $\htt(p) = \gamma +1$,
 	\item $\name{q} = \widehat{h}$ for a function $h \in M$ (with $\dom(h) \subseteq X_\alpha$) mapping into the top level of $p$, i.e. $\ran(h) \subseteq p \cap 2^{\gamma}$.
\end{enumerate}
\ed
\bd Let $\mathbb{R}^\bullet$ to be the subset of $\mathbb{R}$ consisting of  elements $r$ satisfying  $r_\alpha \in \mathbb{R}_\alpha^\bullet$ for every $\alpha \in \supp(r)$.
\ed

\begin{lemma} \label{hulyetechnikas}
	Let $\alpha \in C$. Then $\mathbb{R}_\alpha^\bullet$ is dense in $\mathbb{R}_\alpha$.
\end{lemma}
\begin{proof}
	Since for a fixed $\langle p , \name{q} \rangle \in \mathbb{R}_\alpha$
	\[ p \Vdash \name{q} \in \name{\mathbb{Q}_\alpha} =  \{ f:  \ \dom(f)\in  [\widehat{X_\alpha}]^{<\omega_1}, \ \ran{f}\subseteq \name{T_\alpha}  \} ,\]
	thus
	\[ \begin{array}{rl} p \Vdash & \exists f,g \text{ enumerations of } \dom(q)\subseteq \widehat{X_\alpha},\ \ran(q) \subseteq 2^{<\omega_1}  \text{  in type } \omega , \\
	& \text{where } (\forall i) \  \name{q}(f(i)) = g(i),   \end{array} \]
	and by applying the maximal principle \cite[II., Thm. 8.2]{kunen} two times, there exist  names $f'$ and $g'$ such that
	\begin{equation} \label{q} p \Vdash f' \text{ is an enumeration of } \dom(q)  \text{  in type } \widehat{\omega} ,\end{equation}
	and
	\begin{equation} \label{qq} p \Vdash g' \text{ is an enumeration of } \ran(q)  \text{  in type } \widehat{\omega}, \ \forall i \ (q(f'(i)) = g(i))  .\end{equation}
	Now, recall that $\mathbb{P}_{\alpha} = \mathbb{P}_{\hom}$ is $\omega_1$-closed (by Lemma $\ref{omclos}$),
	thus the set $2^{<\omega_1}$ will not grow by an extension with a $\mathbb{P}_{\alpha}$-generic filter.
	Moreover, the $\omega_1$-closedness of $\mathbb{P}_{\alpha}$ allows us to apply Lemma $\ref{ujfvneve}$, 
	and we obtain a condition $p' \leq p$, and  functions $f_0: \omega \to X_\alpha$,
	$g_0: \omega \to 2^{<\omega_1}$ ($f_0,g_0 \in M$) such that
	\[ p' \Vdash f' = \widehat{f_0}  \ \wedge \ g' = \widehat{g_0}.\]
	By $\eqref{q}$-$\eqref{qq}$ this means that $p'$ determines $\name{q}$, i.e. letting $h = g' \circ f'^{-1}$
	\[ p' \Vdash \name{q} = \widehat{h}. \]
	By extending $p'$	further if necessary, we may assume that $p'$
	has successor height, say
	$\htt(p')= \gamma + 1$, and, moreover, $\ran(h) \subseteq 2^{\leq \gamma}$.
	Now, for each $x \in \dom(h)$,
	choose $s_x \in p' \cap 2^\gamma$	extending $h(x)$ (this is possible since $p'$ is normal),
	and define a function $h'$ with $\dom(h') = \dom(h)$ and $h'(x)= s_x$ for all
	$x \in \dom(h)$. Then we are done, since
	\[ \langle p, \name{q} \rangle \geq \langle p', \name{q} \rangle = \langle p', \widehat{h} \rangle \geq \langle p', \widehat{h}' \rangle \in \mathbb{R}^\bullet_\alpha. \]
\end{proof}

\bfa \label{surufa} For a set $E \subseteq C$ the restriction $\mathbb{R}^\bullet \restrict{E} = \{ r \restrict{E}: \ r \in \mathbb{R}^\bullet \}$ is a dense subset of $\mathbb{R} \restrict{E}$.
\efa

\begin{lemma} \label{cc}
	Let $M' \supseteq M$ be a c.t.m. such that
	\begin{equation} \label{folfelt} M^\omega \cap M' = M^\omega \cap M, \end{equation}
	i.e. there are no new sequences of type $\omega$ consisting of elements of $M$.
	
	Then  for any set $E \subseteq C$ ($E \in M'$)
	\[ M' \models \mathbb{R}\restrict{E} \text{ is } \omega_2 \text{-cc} .\]
\end{lemma}
\begin{proof}
	 Note that our conditions imply that ($(\omega_1)^M = (\omega_1)^{M'}$), and $CH$ holds also in $M'$.
	 For the conclusion of the lemma these corollaries would be sufficient, but in our applications $\eqref{folfelt}$ will always hold.
	
	We will need the following lemma \cite[Ch II. Thm. 1.6.]{kunen} which we will refer to as the $\Delta$-system Lemma.
	\begin{lemma} \label{delta}
		Let $\kappa$ be an infinite cardinal, let $\theta > \kappa$ be regular, and satisfy $\forall \alpha < \theta$ ($|\alpha^{<\kappa}| < \theta$). Assume that $|\mathcal{A}| \geq \theta$, and 
		$\forall x \in \mathcal{A}$ ($|x| < \kappa$). Then there is a $\mathcal{D} \subseteq \mathcal{A}$,
		such that $|\mathcal{D}|= \theta$, and $\mathcal{D}$ forms a $\Delta$-system, i.e.
		there is a kernel set $y$ such that 
		\[ \forall x \neq x' \in \mathcal{D}: \ \ x \cap x' = y. \]
	\end{lemma}

	From now on we will work in $M'$, therefore $\omega_2$ will stand for $\omega_2^{M'}$.
	Assume on the contrary that $A   \subseteq \mathbb{R}\restrict{E}$ is an antichain of size $\omega_2$.
	We can apply Lemma $\ref{delta}$ for the set of supports $\{\supp(r): \ r \in A \}$ of the antichain (with $\kappa = \omega_1$, and $\theta = \omega_2$),
 since each support is countable (by $\eqref{Rdef}$) and $\omega_1^\omega = \omega_1$ in $M'$ (by $\eqref{folfelt}$). Hence we can assume
  w.l.o.g. that $\{\supp(r): \ r \in A \}$ is a $\Delta$-system with the kernel 
 \begin{equation} \label{S}
 S \subseteq E , \end{equation} 
 that is for each $r \in A$, $\alpha \in S$, 
 \[ \pr_\alpha(r) \neq \mathbb{1}_{\mathbb{R}_{\alpha}}, \]
 and if $\alpha \notin S$, then there is at most one $r \in A$ for which 
 \[ \pr_\alpha(r) \neq \mathbb{1}_{\mathbb{R}_{\alpha}}. \]

By Fact $\ref{surufa}$ we can assume that $A \subseteq \mathbb{R}^\bullet \restrict{E}$.
Therefore for each $u \in A$ we can
define the function $h_u$ with	$\dom(h_u) \subseteq \cup \{ X_\alpha: \alpha \in S \}$,
 $\ran(h_{u}) \subseteq 2^{<\omega_1}$, and a sequence $ \langle p^{(u)}_\alpha : \ \alpha \in S \rangle \in  \prod_{\alpha \in S} \mathbb{P}_\alpha$  
such that the following holds
\begin{equation} \label{egen}
 \pr_\alpha(u) = \langle p^{(u)}_\alpha, \widehat{h_{u}\restrict{X_\alpha}} \rangle \ \ (\forall \ \alpha \in S).
\end{equation}

   Since 
   \[ M' \models  |\mathbb{P}_{\hom}| \leq |(2^{<\omega_1})^\omega| = |\omega_1^\omega| = \omega_1 \] by $CH$,
   $|\prod_{\alpha \in S} \mathbb{P}_\alpha | = |\omega_1^\omega| = \omega_1$ (in $M'$), we obtain that there is an element
   $p = \langle p_\alpha: \ \alpha \in S \rangle \in \prod_{\alpha \in S} \mathbb{P}_\alpha$, and a set $D \subseteq A$, $|D| = \omega_2$,
   such that
   \begin{equation} \label{ppp} \forall u \in D: \ \ \langle p_\alpha^{(u)}: \ \alpha \in S \rangle = p. \end{equation}
   
   
    Now we will apply the $\Delta$-system Lemma for the system $\{ \dom(h_u): u \in D \}$ 
  of countable sets, hence there is a subset $|D'| = \omega_2$ such that
  \begin{equation} \label{Del} \{ \dom(h_u):  u \in D' \} \ \text{ is a }\Delta\text{-system with the kernel }K \subseteq \cup  \{X_\alpha: \alpha \in S \}, \end{equation} 
 where $K$
  is countable.  
  Now 
  \[ \{h_{u}\restrict{K} : \ \ u \in D' \} \subseteq (2^{<\omega_1})^K, \]
  where this latter set has size at most $\omega_1$, because by 
  $CH$
  \[ \left|(2^{<\omega_1})^K\right| =  |\omega_1^\omega| =  \omega_1 . \] 
  Therefore we can obtain a subset $D'' \subseteq D'$ of size $\omega_2$ such that 
  \begin{equation} \label{mege} h_{u}\restrict{K} = h_{v}\restrict{K} \text{ for each } u,v \in D'' .\end{equation}
  
  Now it is straightforward to check that if $u \neq v \in D''$, then $u$ and $v$ are compatible.
%

\end{proof}

Next we prove that for each $\alpha \in C$ the  $T_\alpha$ has $|X_\alpha|$-many branches in $M[G_\alpha]$.
With a slight abuse of notation, from now on we will identify each branch $b \in \mathcal{B}(T)$ with
the corresponding function from $\omega_1$ to $\{ 0,1 \}$, i.e. the following holds
\[ \mathcal{B}(T) = \{ f: \omega_1 \to 2: \ (\forall \alpha < \omega_1) f\restrict{\alpha} \in T \} .\]

\begin{lemma}
	Let $\alpha \in C$ be fixed. Then the following holds in $M[G_\alpha]$.
	\[ M[G_\alpha] \models |\mathcal{B}(T_\alpha)| = |X_\alpha|.\]
\end{lemma}
\begin{proof}

    	Using $\eqref{FH}-\eqref{MGFH}$, there is a filter $F_\alpha$  which is $\mathbb{P}_\alpha$-generic over $M$, and 
    a filter $H_\alpha \subseteq \name{\mathbb{Q}_\alpha}[F_\alpha]$ which is $\name{\mathbb{Q}_\alpha}[F_\alpha]$-generic over
    $M[F_\alpha]$. Now by the very definition of the name $\name{\mathbb{Q}_\alpha}$, $\eqref{ite1}-\eqref{ite3}$,
    $ \name{\mathbb{Q}_\alpha}[F_\alpha]$ is the notion of forcing
    \[ \name{\mathbb{Q}_\alpha}[F_\alpha] =  \{ f:  \ \dom(f)\in [X_\alpha]^{<\omega_1}, \ \ran{f}\subseteq T_\alpha  \}, \]
    where a condition $g$ is stronger than $f$ iff for each $x \in \dom(f)$, 
    $g(x) \in T_\alpha$ is an end-extension  of $f(x)$ as functions,
    i.e.,
    \[ g(x)\restrict{\dom(f(x))} = f(x).\]
    It is straightforward to see that a generic filter adds branches, where the pairwise distinct new branches (by genericity) are corresponding to elements of
    $X_\alpha$, thus
    \begin{equation} \label{egyenl1}
    M[G_\alpha] = M[F_\alpha][H_\alpha]  \models |\mathcal{B}(T_\alpha)| \geq |X_\alpha|.
    \end{equation}
    The following lemma proves that  the inequality  $|\mathcal{B}(T_\alpha)| \leq |X_\alpha|$ also holds in $  M[G_\alpha] $.
    
    Before stating that lemma we state that
    for functions $f,g$ where $\ran(f),\ran(g) \subseteq \{0,1 \}$ having the same domain, by $f+g$
    we mean the pointwise addition modulo $2$.
    \begin{lemma} \label{legtech}
    	Denoting the set of branches explicitly added by the filter $H_\alpha$ by 
    	\[ \mathcal{B}_H = \{ b_x : \ x \in X_\alpha \} \in M[F_\alpha][H_\alpha] = M[G_\alpha], \]
    	the following will hold.
    	For each branch $b \in \mathcal{B}(T) \cap  M[G_\alpha]$, there is an ordinal $\beta < \omega_1$, and branches
    	$b_1, b_2, \dots , b_n \in \mathcal{B}_H$ such that
    	\[ b\restrict{\omega_1 \setminus \beta} = (b_1 + b_2 + \dots + b_n)\restrict{\omega_1 \setminus \beta} , \]
    	that is, for each $\gamma \geq \beta$ $b(\gamma) = (b_1 + b_2 + \dots + b_n)(\gamma)$.
    	
    \end{lemma}
    This is a statement in $M[G_\alpha]$, assume that it doesn't hold, and let $b \in \mathcal{B}(T) \cap  M[G_\alpha]$
    be a counterexample, and $\dot{b}$ a name for it. Then there is an element $r \in G_\alpha$ that forces (in $M$) that $b$ is a counterexample, 
    i.e.
    \begin{equation} \label{bforsz} \begin{array}{rl} r \Vdash & (\forall s = \{\name{b_1}, \name{b_2}, \dots,\name{b_n} \} \in [\mathcal{B}_H]^{<\omega})   \\ & (\forall \gamma < \widehat{\omega_1}) (\exists \beta \geq \gamma) \ \ (\name{b_1}+ \name{b_2} + \dots + \name{b_n})(\beta) \neq \dot{b}(\beta)   \end{array} \end{equation}
    By Lemma $\ref{hulyetechnikas}$ we can assume that $r \in \mathbb{R}_\alpha^\bullet$.
    	   
    Working in $M$, first we will need the following Claim.
    \bc \label{sorcl}
    There exist a decreasing sequence of conditions $\langle r_i = \langle p_i, \widehat{h_i} \rangle: \ i \in \omega \rangle \in M$ in $\mathbb{R}^\bullet_\alpha$, and a strictly increasing sequence of countable ordinals , $\langle \gamma_i: \ i \in \omega \rangle \in M$
     such that the following conditions hold
    \begin{enumerate}[(i)]
    	\item \label{fel0} $r_0 = r$ from $\eqref{bforsz}$,
    	\item \label{mas} the height of $p_i$ is $\gamma_i+1$,
    	
    	\item \label{raan} the function $h_i : \dom(h_i) \to p_i$ maps its domain \emph{onto} $p_i$-s top level, i.e.,
    	\[ \ran(h_i) = p_i \cap 2^{\gamma_i}, \]
      	\item \label{utos} for each $i \in \omega$, $k \in \omega$, $s = \{x_0,x_1, \dots, x_{2k} \} \in [\dom(h_i)]^{2k+1}$ there exists
    	$\beta_s \in  [\gamma_i, \gamma_{i+1})$ such that
    	\begin{equation} \label{forsz} r_{i+1} \Vdash  \dot{b}(\beta_s) \neq \left(\name{b_{x_0}} + \name{b_{x_1}} + \dots + \name{b_{x_{2k}}}\right)(\beta_s) .  \end{equation}
    	
    \end{enumerate}
    \ec
    
    \bc \label{korlcl}
	    For the sequences $\left\langle r_i = \langle p_i, \widehat{h_i} \rangle: \ i \in \omega \right\rangle, \langle \gamma_i: \ i \in \omega \rangle \in M$ given by Claim $\ref{sorcl}$
	    there exist a countable ordinal $\gamma_\infty$ and a lower bound $r_\infty = \langle p_\infty, \widehat{h_\infty} \rangle \leq r_i $ ($\forall i \in \omega$) in $\mathbb{R}_\alpha$, where 
	    \begin{equation} \label{gammaoo} \htt(p_\infty) = \gamma_\infty + 1=  \sup \{ \gamma_i: \ i \in \omega \} + 1, \end{equation}
	    and 
	    \[ h_\infty: \bigcup_{i \in \omega} \dom(h_i) \to p_\infty, \]
	    \[ h_\infty(x) = \cup\{h_i(x): \ x \in \dom(h_i) \}. \] 
	    Moreover, for each $t \in \Le_{\gamma_\infty} (p_\infty) = p_\infty \cap 2^{\gamma_\infty}$ there exists 
	    $\delta < \gamma_\infty$, $k \in \omega$, $\{ x_0,x_1, \dots x_{2k}\} \in [\dom(h_\infty)]^{2k+1}$ such that
	    \[ t\restrict{\gamma_\infty \setminus \delta} = \left(h(x_0) + h(x_1) + \dots h(x_{2k}) \right)\restrict{\gamma_\infty \setminus \delta}. \]
	     
    \ec
    
    Before proving these claims first we show that Claim $\ref{sorcl}$ and $\ref{korlcl}$ finish the proof of Lemma $\ref{legtech}$.
    Suppose that Claim $\ref{korlcl}$ gives the lower bound $r_\infty = \langle p_\infty, \widehat{h_\infty} \rangle$ for the decreasing sequence 
    $\left\langle r_i = \langle p_i, \widehat{h_i} \rangle: i \in \omega \right\rangle$ (given by Claim $\ref{sorcl}$). 
    Then for a generic filter $G' \subseteq \mathbb{R}_\alpha$ with $r_\infty \in G'$, $r_\infty$ determines the levels
    of the generic tree $T' = \cup \{ p: \ \exists \name{q} \ \ \langle p, \name{q} \rangle \in G'\}$ below $\htt(p_\infty) = \gamma_\infty+1$, i.e. $T' \cap 2^{\leq \gamma_\infty} = p_\infty$. 
    
    Therefore, towards a contradiction, suppose that $G' \subseteq \mathbb{R}_\alpha$ is an arbitrary filter that is generic over $M$ with $r_\infty = \langle p_\infty, \widehat{h_\infty} \rangle \in G'$, and $\eqref{bforsz}$ holds with $\dot{b}$.
    In $M[G']$
    $\dot{b}[G']: \omega_1 \to 2$ is a branch through the generic tree $T' \supseteq p_\infty$, and
     $\dot{b}[G']\restrict{\gamma_\infty}$ must be an element of $T' \cap 2^{\gamma_\infty} = p_\infty \cap 2^{\gamma_\infty}$. 
	Claim $\ref{korlcl}$ states that there exist $\delta < \gamma_\infty$, $k \in \omega$, $\{ x_0,x_1, \dots x_{2k}\} \in [\dom(h_\infty)]^{2k+1}$  such that
	\beeq \label{begyenl} \dot{b}[G']\restrict{\gamma_\infty \setminus \delta} = \left(h_\infty(x_0) + h_\infty(x_1) + \dots h_\infty(x_{2k}) \right)\restrict{\gamma_\infty \setminus \delta}. \eeq

     Also by the construction of $h_\infty$ (Claim $\ref{korlcl}$)
    $\dom(h_\infty) = \cup \{ \dom(h_i) : i \in \omega \}$, and $\gamma_\infty = \sup \{ \gamma_i : \ i \in \omega \}$ (and recall that $r_i = \langle p_i, \widehat{h_i} \rangle$ is decreasing, hence $\langle \dom(h_i): \ i \in \omega \rangle$ is increasing). This means that there is a finite $n$ such that
    $x_0,x_1, \dots, x_{2k} \in \dom(h_n)$, and $\gamma_n > \delta$. 
    Then condition $\eqref{utos}$ from Claim $\ref{sorcl}$ implies that
    there is a $\beta \in [\gamma_n, \gamma_{n+1})$ such that
    \begin{equation} \label{fff} r_\infty \leq r_{n+1} \Vdash \dot{b}(\beta) \neq \left(h_\infty(x_0) + h_\infty(x_1) + \dots + h_\infty(x_{2k})\right)(\beta). \end{equation}
    In particular (using that $\delta < \gamma_n \leq \beta < \gamma_{n+1} < \gamma_{\infty}$)
    \[ G' \ni r_\infty \Vdash \dot{b}\restrict{\gamma_\infty \setminus \delta} \neq \left(h_\infty(x_0) + h_\infty(x_1) + \dots + h_\infty(x_{2k})\right)\restrict{\gamma_\infty \setminus \delta}, \]
    which contradicts $\eqref{begyenl}$.
    
    For the proof of Claim $\ref{sorcl}$ and $\ref{korlcl}$ we will need the following technical preparations.
    	\bc \label{lezar}
    	Suppose that $\xi < \omega_1$ is a limit ordinal and $T' \subseteq 2^{<\xi}$ is a countable homogeneous normal tree
    	of height $\xi$, and $U \subseteq \mathcal{B}(T') \subseteq 2^\xi$, $U \neq \emptyset$ is a countable set of branches of $T'$.
    	Then 
    	\[ T= T' \cup \left\{ t \cup (u_1+u_2+ \dots+ u_{2k+1})\restrict{\xi \setminus \dom(t)} : \ t \in T', \{u_1,u_2, \dots, u_{2k+1} \} \in [U]^{2k+1}, k \in \omega \right\} \]
    	is a countable homogeneous normal tree of height $\xi +1$, where $T\restrict{\xi} = T'$.
    	\ec 
    	\begin{proof}
    		Define the set $B$ as
    		\begin{equation} \label{Bdf} B =  \left\{ u_1+ u_2 + \dots + u_{2k+1}   : \ k \in \omega, \{u_1,u_2, \dots u_{2k+1}  \} \in [U]^{2k+1} \right\} , \end{equation}
    		and
    		\begin{equation} \label{palak}  T = T' \cup \{ t \cup b\restrict{\xi \setminus \dom(t)}: t \in T', b \in B \}, \end{equation}
    		i.e. we add some branches to $T'$ and obtain a tree of height $\xi+1$, in other words 
    		\begin{equation} \label{megj}  T \setminus T' = \Le_\xi(T) = 2^\xi \cap T. \end{equation}

    		First we have to check that every element of $2^\xi$ which we added is indeed a branch of
    		$T'$, i.e., 
    		\begin{equation} \label{bran} \left(t \cup b\restrict{\xi \setminus \dom(t)}\right)\raisebox{-.5ex}{$\big|_\beta$} \in T' \text{ for each } b \in B, t \in T', \beta < \xi. \end{equation}
    		Fixing an arbitrary element
    		$b=u_1+ u_2 + \dots + u_{2k+1}$ from $B$, $t \in T'$, and $\beta < \xi$ first we can assume that $\beta > \dom(t)$.
    		Observe that for each $\delta \in (\dom(t), \xi)$, using that $T'$ contains $t$, $u_1\restrict{\delta}$, $u_2\restrict{\delta}$, $\dots$, $u_{2k+1}\restrict{\delta}$, by the homogeneity of $T'$
    		\[ b\restrict{\delta} = \left( F_{\left(u_1\restrict{\delta} \right)\left(u_2\restrict{\delta}\right)} \right) \circ \left( F_{\left(u_3\restrict{\delta} \right) \left( u_4\restrict{\delta}\right)} \right) \circ \dots \circ \left( F_{\left(u_{2k-1}\restrict{\delta}  \right) \left( u_{2k}\restrict{\delta}\right)} \right)(u_{2k+1}\restrict{\delta}) \in T' .\]
    		Now, if  $t, b\restrict{\beta} \in T$, 	we can use Claim $\ref{segit}$ to get that $t \cup b\restrict{\beta \setminus \dom(t)} \in T'$.

    		$T$ is obviously countable,
    		and the normality will follow from the fact that $T'$ is normal, we only
    		have to check that for each $t \in T$ there is $t' \in T \cap 2^\xi$ greater than $t$, i.e., $t \subseteq t'$. 
    		Indeed, if $t \in T$ is not on the top level of $T$ then choosing an arbitrary $u \in U \neq \emptyset$,
    		we have by the construction that $t \cup u\restrict{\xi \setminus \dom(t)} \in T$.
    		
    		For the homogeneity of $T$, fix $\beta \leq \xi$, $c,d \in \Le_\beta(T) = 2^\beta \cap T$, $t \in T$, we have to check 
    		that $F_{cd}(t)$ is in $T$. We can assume that 
    		$\dom(t) = \xi$ since otherwise $t \in T'$, and the homogeneity of $T'$ implies that 
    		$F_{cd}(t) = F_{\left(c\restrict{\dom(t)} \right) \left(d\restrict{\dom(t)} \right)}(t) \in T'$.
    		Therefore $\dom(t) = \xi$, and $t = t' \cup b\restrict{\xi \setminus \dom(t)}$ for some $t' \in T'$, $b \in B$. 
    		Second, if $\beta = \dom(c) = \dom(d) < \xi$, then letting $\delta = \max\{\beta, \dom(t') \}$, $t$ can be considered as
    		\[ t = t\restrict{\delta} \cup b\restrict{\xi \setminus \delta}, \]
    		(where $t\restrict{\delta} \in T'$ by $\eqref{bran}$), hence again by the homogeneity of $T'$ we have
    		\[ F_{cd}(t) = F_{cd}(t\restrict{\delta}) \cup b\restrict{\xi \setminus \delta} \in T. \]
    		This means that the only remaining case is when $\beta = \xi$, that is, $c,d, \in 2^\xi$, and are of the form
    		\[ c = t'' \cup (b'')\restrict{\xi \setminus \dom(t'')} \text{ for some } b'' \in B, \]
    		\[ d = t''' \cup (b''')\restrict{\xi \setminus \dom(t''')} \text{ for some } b''' \in B .\]
    		Now, if the ordinals $\dom(t'),\dom(t''),\dom(t'') \in \xi$ are not equal, then letting $\delta = \max \{ \dom(t'),\dom(t''),\dom(t'') \}$, we can view
    		$t = t' \cup b\restrict{\xi \setminus \dom(t')}$ as $t= t\restrict{\delta} \cup b\restrict{\xi \setminus \delta}$, and similarly 
    		\[ c = c \restrict{\delta} \cup (b'')\restrict{\xi \setminus \delta} , \]
    		\[ d = d\restrict{\delta} \cup (b''')\restrict{\xi \setminus \delta}.\]
    		
    		Then 
    		\[ F_{cd}(t) = F_{\left( c \restrict{\delta}\right) \left( d\restrict{\delta} \right)}(t\restrict{\delta}) \cup (b+c+d)\restrict{\xi \setminus \delta},\]
    		we would only need that $b+c+d \in B$. By $\eqref{Bdf}$ $b,c,d \in B$ implies $b+c+d \in B$, 
    		therefore $T$ is a homogeneous tree, indeed.
    		
    	\end{proof}
    	Moreover, we obtain the following.
    	\begin{cor} \label{zrtsgfele}
    		Let $\delta \in C$, $\langle v_i: \ i \in \omega \rangle \in M$ be a decreasing sequence in $\mathbb{R}^\bullet_\delta$. 
    		Then there is a common lower bound $v_\infty = \langle w_\infty, \widehat{g_\infty} \rangle \in \mathbb{R}^\bullet_\delta$ of the sequence.
    		
    	\end{cor}	
    	\begin{proof}
    		Let $w_n$, $g_n$ are such that $v_n = \langle w_n, \widehat{g_n} \rangle$.
    		Let $w = \cup \{ w_n: \ n \in \omega \}$ (which is in $\mathbb{P}_\delta = \mathbb{P}_{\hom}$ by Lemma $\ref{omclos}$), and $g_\infty$ to be the function such that $\dom(g_\infty) = \cup \{ \dom(g_n): \ n \in \omega \}$, 
    		assigning $g_\infty(x) = \cup \{ g_n(x): \ x \in \dom(g_n) \}$. Then $g_\infty(x)$ is a branch of $w$.
    		
    		If the $w_n$'s  are strictly decreasing (and thus the $\htt(w_n)$'s are strictly increasing)  $\htt(w)$ must be a limit ordinal, and then for obtaining $w_\infty$ we can apply Claim $\ref{lezar}$ with 
    		\[ U = \{ g_\infty(x): \ x \in \dom(g_\infty)\}, \] 
   		$\xi = \htt(w)$, $T'= w$.
   		

    	\end{proof}
	    \bcor \label{bccc}
	     For $\delta \in C$, $r = \langle p, \widehat{g} \rangle \in \mathbb{R}^\bullet_\delta$, with the countable ordinals $\htt(p) = \xi \leq \xi'$ there exists an extension $r' = \langle p', \widehat{g'} \rangle \in \mathbb{R}^\bullet_\delta$ of $r$ with $\htt(p')  = \xi' +1$.
	     \ecor
	     \bp
		      The statement holds also for $\xi'<\xi$, if  $\xi' + 1 = \xi$. We apply induction on $\xi'$, and assume that $\xi' \geq \xi$.
		     

		     For $ \xi' = \xi'' +1$, if $r'' = \langle p'', \widehat{g''} \rangle \leq r$ is the desired 
		     extension for $\xi''$, then let $p' = p'' \cup \{ t \tieconcat i: \ i \in \{0,1 \}, t \in p \}$, and for each $x \in \dom(g'')$ let $g'(x) = g(x) \tieconcat 0$. 
		     
		     Finally, for limit $\xi' > \xi$ choose a sequence $\xi_0 = \xi < \xi'_1 < \dots < \xi'_n < \dots$ such that $\sup \{\xi'_i: \ i \in \omega\} = \xi'$. By induction choose a decreasing sequence 
		     \[ r = \langle p, \widehat{g} \rangle \geq \langle p'_1, \widehat{g'_1} \rangle \geq \dots \langle p'_i, \widehat{g'_i} \rangle \geq \dots \]
		     in $\mathbb{R}_\delta$
		     so that $\htt(p'_i) = \xi'_i + 1$.
		     Now applying Claim $\ref{zrtsgfele}$ will work.
	     \ep
    
    \bp(Claim $\ref{sorcl}$)
   We are given $r_0= \langle p_0, \widehat{h_0} \rangle = r \in \mathbb{R}_\alpha^\bullet$ (the height of $p_0$ $\htt(p_0) = \gamma_0+1$), and we will apply induction. \\
   Assume that $r_0, r_1, \dots, r_{i}$, and $\gamma_0, \gamma_1, \dots, \gamma_{i}$ are defined.
   Let $\langle s_n: \ n \in \omega, \rangle $ be an enumeration of the set $\cup \{[\dom(h_{i})]^{2k+1} : \ k \in \omega \}$ (recall that $\dom(h_i)$ is countable using $\eqref{ite1}$ from Definition $\ref{Qdef}$). Now we construct a decreasing
   sequence
   \[ r_i \geq v_0 \geq v_1 \geq \dots \geq v_m \geq \dots \]
   below $r_i$ in $\mathbb{R}^\bullet_\alpha$, and
   a sequence $\langle \beta_n: \ n \in \omega \rangle$ (where each  $\beta_j \geq \gamma_i$)
   such that 
   \begin{enumerate}
   
   	\item if $s_m = \{ x_0, x_1, \dots, x_{2k}\}$, then 
   	\begin{equation} \label{mforsz} v_m \Vdash \dot{b}(\beta_m) \neq (\name{b_{x_0}}+ \name{b_{x_1}} + \dots + \name{b_{x_{2k}}})(\beta_m) \end{equation}
   		\item for $v_m = \langle w_m, \widehat{g_m} \rangle$ we have $\htt(w_m) > \beta_m +1$.

   \end{enumerate}
   Suppose that the $\beta_j$-s, and the $v_j$-s are defined for $j \leq l$, and let
   $s_{l+1} = \{x_0,x_1, \dots, x_{2k} \}$.
   Then using that $v_l \leq r_i \leq r$ and $\eqref{bforsz}$, 
   \begin{equation}   v_l \Vdash   \ (\exists \beta \geq \gamma_i) \ (\name{b_{x_1}}+ \name{b_{x_2}} + \dots + \name{b_{x_{2k}}})(\beta) \neq \dot{b}(\beta),  \end{equation}
   hence there is a countable ordinal $\beta_{l+1} \geq \gamma_i$, and a condition $v'_l \leq v_l$,
   such that
   \beeq \label{betalp} v'_l \Vdash   \  (\name{b_{x_0}}+ \name{b_{x_1}} + \dots + \name{b_{x_{2k}}})(\beta_{l+1}) \neq \dot{b}(\beta_{l+1}). \eeq
   By Lemma $\ref{hulyetechnikas}$, we can assume that $v'_l \in \mathbb{R}_\alpha^\bullet$, and let $v_{l+1} = v'_l = \langle w'_l, \widehat{g'_l} \rangle$.
   Also by further extension (using Corollary $\ref{bccc}$) we can assume that 
	\beeq \label{htw} \htt(w_{l+1}) \text{  is  greater than }\beta_{l+1} +1. \eeq

   As we obtained the decreasing sequence $\langle v_j =  \langle w_{j}, \widehat{g_{j+1}} \rangle : \ j \in \omega \rangle$ under $r_i = \langle p_i, \widehat{h_i} \rangle$
   and the $\beta_j$-s, we can define  $r_{i+1} = \langle p_{i+1}, \widehat{h_{i+1}} \rangle \in \mathbb{R}_\alpha^\bullet$ to be a lower bound of the $v_j$'s as follows.
   Using Corollary $\ref{zrtsgfele}$ first we define $\langle p'_{i+1}$, $\widehat{h'_{i+1}} \rangle$ to be a lower bound of the $v_j$'s.
   
  Define $h_{i+1}$ as follows. For each $x \in \dom(h'_{i+1})$ let $h_{i+1}(x)=  h'_{i+1}(x)$.
  For ensuring $\eqref{raan}$, for each $t \in (p_{i+1} \cap 2^{\gamma_{i+1}}) \setminus \{ h_{i+1}(x): \ x \in \dom(h_{i+1}) \}$
  we can pick pairwise distinct elements $x_t$ from $X_\alpha \setminus \dom(h'_{i+1})$, and define $h_{i+1}(x_t) = t$. Now we have checked $\eqref{raan}$.
 
   It remained to check that $\gamma_{i+1}$, $p_{i+1}$, $h_{i+1}$ (defined by the equalities $r_{i+1} = \langle p_{i+1}, \widehat{h_{i+1}} \rangle$ and $\htt(p_{i+1}) = \gamma_{i+1} +1$) satisfy  $\eqref{utos}$.
   If $s_m = \{ x_1,x_2, \dots, x_{2k+1} \} \in [\dom(h_i)]^{2k+1}$, then $r_{i+1} \leq v_{m+1} \leq v'_m$ and 
   $\eqref{betalp}$ together implies $\eqref{forsz}$ from $\eqref{utos}$. $\beta_{m} \in [\gamma_i, \gamma_{i+1})$ follows from the fact that $\eqref{htw}$ holds for $v_{l+1}$'s first coordinate $w_{l+1}$.
   \ep

   \bp(Claim $\ref{korlcl}$) 
   So suppose that  $\langle r_i = \langle p_i, \widehat{h_i} \rangle: \ i \in \omega \rangle$, $\langle \gamma_i: \ i \in \omega \rangle$ fulfills our
   requirements $\eqref{fel0}-\eqref{utos}$. Let $p' = \cup \{p_i: \ i \in \omega \}$ which is a 
   countable homogeneous normal tree of height 
   \begin{equation} \label{gam} \gamma= \sup \{ \gamma_i : \ i \in \omega \} \end{equation}
   by Lemma
   $\ref{omclos}$, and because 
   the sequence of $\gamma_i$-s is strictly increasing. We define the function $h$ as follows. 
   \begin{equation} \label{domh} \dom(h) = \cup \{\dom(h_i): \ i \in \omega \}, \end{equation}
   and for each $x \in \dom(h)$ define $h(x)$ to be $\cup \{ h_i(x): \ x \in \dom(h_i) \}$, which is a function, since
   $\langle p_i, \widehat{h_i} \rangle$-s form a decreasing sequence in $\mathbb{R}^\bullet_\alpha$. By 
   $\eqref{raan}$, $h(x) \in 2^\gamma$, which is not an element of $p'$, since   $p' \subseteq 2^{<\gamma}$ (by $\htt(p') = \gamma$).

		Apply Claim $\ref{lezar}$ with $T' = p'$, $\xi = \gamma$, $U = \ran(h)$, and let $p_\infty = T$ be the given tree (that is
		\begin{equation} \label{palak2}  p_\infty = p' \cup \{ t \cup b\restrict{\xi \setminus \dom(t)}: t \in p', b \in B \}, \end{equation}
		where 
		\begin{equation} \label{Bdf2} B = \{ h(x_0) + h(x_1) + \dots + h(x_{2k})   : \ k \in \omega, \{x_0,x_1, \dots x_{2k}  \} \in [\dom(h)]^{2k+1} \}), \end{equation}
		we are done.
		\ep

\end{proof}

We will need the following basic lemmas.

\begin{lemma} \label{nincsuj3}
	Suppose that $E \subseteq C$, $E \in M'$ is a set, $M' \supseteq M$ is a c.t.m., such that
	\begin{equation} \label{feltetel}  M^\omega \cap M' = M^\omega \cap M, \end{equation}
	i.e., there is no new sequence of type $\omega$ consisting of elements of $M$.
	Let $T \in M'$ ($T \subseteq 2^{<\omega_1}$) be a tree of height $\omega_1$ with countable levels.
	Then extending $M'$ by a filter $F \subseteq \mathbb{R}\restrict{E}$ which is $\mathbb{R}\restrict{E}$-generic adds no new branches to $T$, i.e.
	\[ \mathcal{B}(T) \cap M'[F] = \mathcal{B}(T) \cap M'. \]
\end{lemma}
\begin{proof}
	As $\mathbb{R}^\bullet\restrict{E}$ is dense in  $\mathbb{R} \restrict{E}$ (by Corollary $\ref{surufa}$), forcing with one yields exactly the same extensions as forcing with the other (by Lemma $\ref{surube}$), we only have to show that $\mathbb{R}^\bullet\restrict{E}$ is $\omega_1$-closed in $M'$ to apply Lemma $\ref{nemadujagat}$.
	Our conditions together with by Corollary $\ref{zrtsgfele}$ imply that for each decreasing sequence (of type $\omega$) in $\mathbb{R}^\bullet\restrict{E}$ belonging to $M'$ has a lower bound. (In fact first we can find such a lower bound only in $\mathbb{R}^\bullet$ and we can restrict the obtained condition, or we can also refer to the fact that these partial functions are countably supported). Then Lemma $\ref{nemadujagat}$ gives the desired result.
	
\end{proof}

\begin{lemma} \label{nincsuj2}
	Let $E \subseteq C$ be a set, $M' \supseteq M$ be a c.t.m. such that
	\beeq \label{folflt} M^\omega \cap M' = M^\omega \cap M, \eeq
	i.e. there is no new sequence of type $\omega$ consisting of elements of $M$.
	Then extending $M'$ by a filter $F \subseteq \mathbb{R}\restrict{E}$ which is $\mathbb{R}\restrict{E}$-generic adds no new sequences of type $\omega$
	consisting of elements of $M'$, i.e.
	\[ (M')^\omega \cap M' = (M')^\omega \cap M'[F]\]
\end{lemma}
\bp
Again, (similarly to the proof of Lemma $\ref{nincsuj3}$) we have that each decreasing $\omega$-sequence in $\mathbb{R}^\bullet\restrict{E}$ belonging to $M'$ has a lower bound (by Corollary $\ref{zrtsgfele}$). Then apply Corollary $\ref{nincsuj}$.
\ep

Recall that $G \subseteq \mathbb{S}$ is $\mathbb{S}$-generic over $M$. In the next lemma we will prove that
if $\alpha \notin C$, then there is no Kurepa tree in $M[G]$ with $\omega_\alpha^{M[G]}$ branches.
\begin{lemma} \label{alfas}
	Let $T \in M[G]$ ($T \subseteq 2^{<\omega_1}$) be a tree of height $\omega_1$ with countable levels, and let $\alpha$ be an ordinal so that
	\beeq \label{alphadef} M[G] \models |\mathcal{B}(T)| = \omega_\alpha. \eeq
	Then $\alpha \in C$.
\end{lemma}
The proof of this lemma will take a lot of effort. From Lemma $\ref{kappacc}$ to Lemma $\ref{xkv}$ we 
will find two models, each containing $T$, but exactly the greater containing all branches of $T$.
From Claim $\ref{izomc}$ to  Lemma $\ref{utso}$ we will see that the homogeneity of our generic $T_\delta$-s imply that the larger of the two models cannot contain all the branches, contradicting our previous arguments.

Fix an ordinal $\alpha$ such that
\begin{equation} \label{alfdef} \begin{array}{c} \alpha \notin C, \text{ and} \\
M[G] \models |\mathcal{B}(T)|] = \omega_\alpha. \end{array} \end{equation}
We will derive a contradiction by finding a suitable intermediate model containing  $M[G] \cap \mathcal{B}(T)$, arguing that the residual forcing still adds new branches.

First we would like to find an intermediate extension $N$ ($M \subseteq N \subseteq M[G]$) which is small enough, but $T \in N$. 

Before that we prove that $T \subseteq M$.

\begin{cl} \label{sorozatt}
	\[ M^\omega \cap M[G] = M^\omega \cap M, \]
	This implies that $\omega_1$ does not collapse, that is,
	\[ M[G] \models \nexists \text{bijection between } \omega \text{ and } \omega_1^M. \]
	Moreover,
	\[ 2^{<\omega_1} \cap M = 2^{<\omega_1} \cap M[G], \]
	in particular
	\[ T \subseteq M. \]
\end{cl}
\bp
We have two cases depending on $C$. The filter $G$ is $\mathbb{S}$-generic, where either $\mathbb{S} = \mathbb{L} \times \mathbb{R}$ (if $2 \notin C$), or
$\mathbb{S} = \mathbb{R}$ (otherwise, $\eqref{Rdf}$).
Now as $\mathbb{L}$ is $\omega_1$-closed, both $\mathbb{R}^\bullet$ and $\mathbb{R}^\bullet \times \mathbb{L}$ are $\omega_1$-closed (Corollary $\ref{zrtsgfele}$ and each condition is countably supported). This means that by Corollary $\ref{nincsuj}$ forcing with $\mathbb{R}^\bullet \times \mathbb{L}$, or $\mathbb{R}^\bullet$ does not add new sequences.
Then recalling Corollary $\ref{surufa}$ we obtain that $\mathbb{R}^\bullet \times \mathbb{L}$ is dense in $\mathbb{L} \times \mathbb{R}$, and  $\mathbb{R}^\bullet$ is dense in $\mathbb{R}$, we are done (by Lemma  $\ref{surube}$).
\ep

\begin{lemma} \label{kappacc}
	Suppose that $M' \supseteq M$ be a c.t.m. where
	\beeq \label{ome} M \cap M^\omega = M' \cap M^\omega, \eeq
	and for our inaccessible $\kappa$ from $\eqref{limzart}$ (and $\eqref{Levy}$)
	\[ (\kappa \text{ is a cardinal})^{M'}. \]
	 Then
	\[ M' \models \mathbb{L} \text{ is } \kappa\text{-cc}. \]
\end{lemma}
\begin{proof}
	Suppose that $A \subseteq \mathbb{L} \in M$ is a an antichain of size $\kappa$. First we can apply the $\Delta$-system lemma (Lemma $\ref{delta}$) for the
	system $\{ \dom(a): \ a \in A \}$, since $\dom(a)$ is countable by $\eqref{Levy}$, and for any infinite ordinal  $\gamma < \kappa$ 
	\[ (\gamma^\omega)^{M'} = (\gamma^\omega)^{M} \leq (\gamma^\gamma)^M < \kappa \]
	by the fact that $\kappa$ is inaccessible in $M$. 
	Therefore we can assume that $\{ \dom(a): \ a \in A \}$ is a $\Delta$-system, let $K \subseteq \kappa \times \omega_1$ denote its kernel. Since $K$ is countable, and $\kappa > \omega$ is inaccessible in $M$, $\eqref{ome}$ implies that $\omega < \cf(\kappa)^{M'}$.
	Therefore, there is an ordinal $\delta < \kappa$
	such that $K \subseteq \delta \times \omega_1$. This and the definition of $\mathbb{L}$ $\eqref{Levy}$ imply that for each $a \in A$, $\ran(a\restrict{K}) \subseteq \delta$.
	But the derived system $A' = \{ a\restrict{K}: \ a \in A \} \subseteq \delta^K$, gives an upper bound 
	\[ (\delta^{\omega})^{M'} = (\delta^{\omega})^{M} \leq (\delta^\delta)^M < \kappa \] for the cardinality of $A'$. This contradicts the fact that $A$ is an antichain of size $\kappa$.
\end{proof}

\begin{lemma} \label{GCH0}
	
		In the final model, $M[G]$
		\beeq \label{elsoresz} M[G] \models |X_\delta| = \omega_\delta \ \ (\forall \delta \in C). \eeq
		In general, cardinals and cofinalities greater than or equal to $(\omega_2)^{M[G]}$ are preserved, where 
		\begin{itemize}
			\item $(\omega_2)^{M[G]}=\kappa$, if $2 \notin C$, that is, we forced with $\mathbb{L}$ too,
			\item $(\omega_2)^{M[G]}=(\omega_2)^M$, if $2 \in C$.
		\end{itemize}
\end{lemma}
\bp	For proving $\eqref{elsoresz}$, by Definition $\ref{Xhalmazok}$ it is enough to show that if $2 \notin C$ then only cardinals strictly between $\omega_1$ and $\kappa$ are collapsed, and if $2 \in C$, then no cardinals are collapsed.

In both cases, Corollary $\ref{sorozatt}$ states that $\omega_1$ is not collapsed. Now if $2 \notin C$, then we had forced with 
$\mathbb{L} \times \mathbb{R}$, otherwise only with $\mathbb{R}$.
In the first case, by Lemma $\ref{szorz2}$, $G$ can be identified with $I \times G\restrict{C}$, and we can consider this extension as first extending with $I$, and then with  $G\restrict{C}$.
Therefore, in both case it is enough to show that
\begin{enumerate}[(1)]
	\item adding the filter $I$ which is generic over $M$ destroys exactly cardinals in $(\omega_1,\kappa)$.
	\item \label{amas}  extending $M$  (resp., $M[I]$) by $G\restrict{C}$ doesn't collapse cardinals greater than $\omega_2^M = \omega_2$ (resp., $\omega_2^{M[I]} = \kappa$)
\end{enumerate}  
Note that by Corollary $\ref{sorozatt}$, $\omega_1$ is absolute.
For the first claim, $\mathbb{L}$  collapses every cardinal between $\omega_1$ and $\kappa$,
because the generic filter gives surjections from $\omega_1$ onto each $\mu < \kappa$. Lemma $\ref{kappacc}$ gives that $\mathbb{L}$
is $\kappa$-cc in $M$, thus by Lemma $\ref{cc-cf}$ cardinals and cofinalities greater than or equal to $\kappa$ remain cardinal in $M[I]$.

For $\eqref{amas}$, we can apply Lemma $\ref{cc}$ for $\mathbb{R}$ (with $M'=M$, and $M'=M[I]$ too, because of Corollary $\ref{sorozatt}$), and we obtain that $\mathbb{R}$ is $\omega_2^{M'}$-cc in $M'$ in each case. Then Lemma $\ref{cc-cf}$ implies that 
cardinals (and cofinalities) greater than or equal to $\omega_2^{M'}$ are still cardinals (and cofinalities) after forcing.
This completes the proof of $\eqref{elsoresz}$.
\ep

\bl \label{GCH}		
		
	 If $M \subseteq M[J] \subseteq M[G]$ is a forcing extension, ($J \subseteq \mathbb{O}$ is generic over $M$), where 
	$\mathbb{O}$ is a notion of forcing (smaller than
			$\omega_2^{M[G]}$)$^M$, then $GCH$ holds in $M[J]$ for $ \nu \geq \omega_2^{M[G]}$.
			
			Moreover, in the case when we used the inaccessible cardinal $\kappa$ (i.e. $(\omega_2)^{M[G]}=\kappa$), then
		\[	  M[J] \models  \ \text{``} \kappa \text{ is inaccessible, in particular } 2^{\omega_1} < \kappa. \text{''}  \]
		
\el

\begin{proof}

	Let $\lambda = |\mathbb{O}|^M < \omega_2^{M[G]}$. By Lemma $\ref{hatvanyh}$, if $\nu$ is a cardinal in $M$, then
	\[ M[J] \models 2^\nu \leq (\lambda^{\nu \cdot \lambda})^M. \]
	This yields that
	\begin{equation} \label{nagyobb} (\nu \geq \lambda) \rightarrow \left( M[J] \models 2^\nu \leq (\nu^{\nu \cdot \nu})^M = (\nu^+)^M \right) \end{equation}
	by $GCH$ in $M$.
		Therefore we also obtain that in $M[J]$ $2^\nu = \nu^+$ for $\nu \geq \omega_2^{M[G]}$ (because $\lambda < \omega_2^{M[G]}$). 
		
		Moreover,	if
		$M[G]= M[I][G\restrict{C}]$ and $\omega_2^{M[G]} = \kappa$ is inaccessible in $M$, then  $\kappa$ is still a strong limit in $M[J]$.
	In this case, because $|\mathbb{O}|^M = \lambda < \kappa$, and $\mathbb{O}$ obviously has the $\kappa$-cc in $M$,
	we have that Lemma $\ref{cc-cf}$ guarantees that $\cf(\kappa)$ is still $\kappa$ in $M[J]$. This yields the conclusion that $\kappa$ remains inaccessible in $M[J]$.
\end{proof}

 For finding our desired model $N$ which contains $T$ as an element, but cannot contain all of its branches (because $2^{\omega_1}$ is smaller there), 
 we need to extend $M$ by filters containing less information than what  $I \subseteq \mathbb{L}$ and $G\restrict{C} \subseteq \mathbb{R}\restrict{C} =\mathbb{R} $ give.
 First we are to find a model $M''$ between $M$ and $M[I]$ extracting minimal information from the extension by $I$.

 We would like to consider the notion of forcing $\mathbb{L}$ as a product.
 \bd For $\mathbb{L}$ defined in $\eqref{Levy}$, and a set of ordinals $K \subseteq   \kappa$ let
 \[ \begin{array}{l} \mathbb{L}\restrict{K} = \{ f \in \mathbb{L}: \ \dom(f) \subseteq K \times \omega_1\} = \\
 = \{ f: \ \dom(f) \subseteq K \times \omega_1, \ |\dom(f)| < \omega_1, \ f(\lambda, \alpha) < \lambda \ \ (\forall \ \lambda \in K \} \end{array}.\]
 Then clearly
 \[ \mathbb{L} \simeq \mathbb{L}\restrict{K} \times \mathbb{L}\restrict{\kappa \setminus K}. \]
 Furthermore, for any filter $F \subseteq \mathbb{L}$ define
 \[ F\restrict{K} = F \cap \mathbb{L}\restrict{K}. \]
 \ed
 
	\bc
	There exists an ordinal  $\mu < \kappa$ such that
	\[ T \in M[G\restrict{C}][I\restrict{\mu}]. \]
	\ec
\bp
	First, using Lemma $\ref{nincsuj2}$
	\begin{equation} \label{nuncs} M^\omega \cap M[G\restrict{C}] = M^\omega \cap M, \end{equation}
	also implying that $\omega_1$ is absolute. Moreover, $\kappa$ is a cardinal in $M[G\restrict{C}] \subseteq M[G]$, because $\mathbb{R} = \mathbb{R} \restrict{C}$ has the $\omega_2$-cc in $M$ (by Lemma $\ref{cc}$).
	Now Lemma $\ref{kappacc}$ states that 
	\begin{equation} \label{cc-c} M[G\restrict{C}] \models \mathbb{L} \text{ is } \kappa\text{-cc. } \end{equation} 
	Applying Lemma $\ref{nice}$ in $M[G\restrict{C}]$, there is a nice $\mathbb{L}$-name $\sigma \in M[G\restrict{C}]$ for a subset of $2^{<\omega_1}$ for which
	\[ \mathbb{1}_\mathbb{L} \Vdash (\dot{T} \subseteq \widehat{2^{<\omega_1}}) \rightarrow (\dot{T} = \sigma) , \]
	where  $\dot{T} \in M[G\restrict{C}]$ is a $\mathbb{L}$-name for $T \in M[G\restrict{C}][I]$.
	Here $\sigma = \{ \{ \widehat{f} \} \times A_f : \ f \in 2^{<\omega_1} \}$ (where each $A_f$ is an antichain in $\mathbb{L}$, and of size $<\kappa$ by $\eqref{cc-c}$).
	Note that $\cf^{M[G\restrict{C}]}(\kappa) = \kappa$ by Lemma $\ref{cc-cf}$ (because $\mathbb{R}\restrict{C}$ is $\omega_2$-cc by Lemma $\ref{cc}$).	
	This means that for each $f \in 2^{<\omega_1}$ there is an ordinal $\mu_f < \kappa$ such that
	\[ \forall l \in A_f: \ \dom(l) \subseteq \mu_f \times \omega_1. \]
	Define
	\beeq \label{mu} \mu = \sup \{ \mu_f : \ f \in 2^{<\omega_1} \} < \kappa. \eeq
	Then clearly $A_f \subseteq \mathbb{L}\restrict{\mu}$ for each $f \in 2^{<\omega_1}$.
	Since $\mathbb{L} \simeq \mathbb{L}\restrict{\mu} \times \mathbb{L}\restrict{\kappa  \setminus \mu}$, and 
	if $I\restrict{\mu} = I \cap \mathbb{L}\restrict{\mu}$, $I\restrict{\kappa \setminus \mu} = I \cap \mathbb{L}\restrict{\kappa \setminus \mu}$ are filters given by $I$ in the components, then the tree $T = \sigma[I]$ depends only on coordinates in $\mu \times \omega_1$, i.e. on $I\restrict{\mu}$. 
	Therefore 
		\[ T = \sigma[I] = \sigma[I\restrict{\mu}] \in M[G\restrict{C}][I\restrict{\mu}], \]
	as desired.

\ep	
	
\bd \label{M''}
			In the case when $G = I \times G\restrict{C}$ (because $2 \notin C$) we define $M'' = M[I\restrict{\mu}]$,
			and if $G = G\restrict{C}$, then let $M'' = M$.
	
\ed

	
	Note that in each case 
	\begin{equation} \label{benn}
	 T \in  M''[G\restrict{C}].
	\end{equation}
	
	\bc \label{Sclaim}
		There exists a set $S \subseteq C$, $S \in M''$ such that
		\[ M'' \models |S| < \omega_2^{M[G]}, \]
		and
		\[ T \in M''[G\restrict{S}]. \] 
	
	\ec
	\bp
	First, since $\mathbb{L}_\mu$ is $\omega_1$-closed (in $M$), $M^\omega \cap M[I\restrict{\mu}] = M^\omega \cap M$. Therefore one can apply
	Lemma $\ref{cc}$ and obtain that in $M$ (and $M[I\restrict{\mu}]$, resp.). $\mathbb{R}\restrict{C}$ has no antichain of size $\omega_2^{M}$ ($\omega_2^{M[I_\mu]}$, resp.). 
		We get that
	\begin{equation} \label{2cc}
	 M'' \models \mathbb{R}\restrict{C} \text{ is } \omega_2 \text{-cc.} 
	\end{equation}
	By Lemma $\ref{nice}$, there is a nice $\mathbb{R}\restrict{C}$-name $\sigma$ for a subset of $\widehat{2^{<\omega_1}}$ such that
	\[ \mathbb{1}_{\mathbb{R}} \Vdash   (\dot{T} \subseteq \widehat{2^{<\omega_1}}) \rightarrow ( \dot{T} = \sigma) \]
	(where $\dot{T} \in M''$ is a $\mathbb{R}\restrict{C}$-name for $T \in M''[G\restrict{C}]$, and $\sigma = \bigcup_{f \in 2^{<\omega_1}} \{ \widehat{f} \} \times A_f$).
	Define $S' = \cup \{A_f : \ f \in 2^{<\omega_1} \} $, and 
	\[ S = \{\supp(r) : \ r \in S'  \} \subseteq C. \]
	Since each $\supp(r)$ is countable (by the very definition of $\mathbb{R}$ $\eqref{Rdef}$, and $\mathbb{R}\restrict{C}$ is a projection)
	and because  $|S'| \leq \omega_1$ (in $M''$), we have that 
	\begin{equation} \label{Smeret}
		M'' \models	 |S| \leq \omega_1.
	\end{equation}
	(Note that, since we worked in $M''$ we only have that				 $S \in M''$.)

Recall that $G\restrict{C}$ can be identified with the product $G\restrict{S} \times G\restrict{C \setminus S}$.
    Now $\sigma[G\restrict{C}]$ depends only on $G\restrict{C}$'s projection onto $\mathbb{R}\restrict{S}$,
     $G\restrict{S}$, and there is a corresponding $\mathbb{R}\restrict{S}$-name $\sigma' \in M''$ such that
     \[  T = \sigma[G\restrict{C}] =\sigma'[G\restrict{S}]  \in M''[G\restrict{S}]. \]
     
   	\ep

 In the beginning of Lemma $\ref{alfas}$, our condition was equality
 \[ M[G] \models |\mathcal{B}(T)| = \omega_\alpha, \]
 (where $\alpha \notin C$),  
  and our goal is to find a model $N$ between $M''$ and $M''[G\restrict{S}]$ with $T \in N$, and 
 \[ N \models 2^{\omega_1} < (\omega_\alpha)^{M[G]},  \]
 implying that
 \[ N \models |\mathcal{B}(T)| < (\omega_\alpha)^{M[G]} \]
 (since each branch corresponds to a function from $\omega_1$ to $2$).
	
	Now working in $M''$, we are to show that $\mathbb{R}\restrict{S}$, which is a product of two-step iterations
	is isomorphic to a two-step iteration of products.	
	 $\mathbb{R}\restrict{S}$ was a product restricted to the countably supported elements, each coordinate is a
	 two-step iteration $\mathbb{P}_\gamma \ast \name{\mathbb{Q}_\gamma}$. 
	  Recall that an element $r \in \mathbb{R}\restrict{S} \subseteq \prod_{\gamma \in S} \mathbb{R}_\gamma$ has coordinates of the form
	   $r_\gamma = \langle p_\gamma, \name{q_\gamma} \rangle \in \mathbb{R}_\gamma$ ($\gamma \in S$). Here $\name{q_\gamma}$ is a $\mathbb{P}_\gamma$-name for which
	 \[ p_\gamma \Vdash_{\mathbb{P}_\gamma} \name{q_\gamma} \in \name{\mathbb{Q}_\gamma} = \{ f:  \ \dom(f)\subseteq \ \widehat{X_{\alpha}}, \ |\dom(f)|<\omega_1, \ \ran{f}\subseteq \name{T_\alpha}  \} \]
	 	 by $\eqref{iter}$, $\eqref{Qdef}$.
	 \bd \label{PSdef}
		  Let $\mathbb{P}\restrict{S}$ be
	 \[ \mathbb{P}\restrict{S} = \{ p \in \prod_{\delta \in S} \mathbb{P}_\delta: \ |\supp(p)|< \omega_1 \}.  \]
	 \ed

	 We will construct a partial order that has a dense subset isomorphic to $\mathbb{R}^\bullet\restrict{S}$. 
	  	
	   \bd	 \label{QSd}
		   Define $\name{\mathbb{Q} \restrict{S}} \in M$ to be the $\mathbb{P}\restrict{S}$-name so that
		   \[ \mathbb{1}_{\mathbb{P}\restrict{S}} \Vdash \ \   \name{\mathbb{Q} \restrict{S}} = \{ f: \ f \text{ is a function,} \dom(f) \in [S]^{<\omega_1}, (\forall \alpha \in \dom(f)) f(\alpha) \in \name{\mathbb{Q}_\alpha}  \}.  \]
	   \ed
	 
	 \bd Let $(\mathbb{P}\restrict{S} \ast \name{\mathbb{Q}\restrict{S}})^\bullet$ be the following subset of the two-step iteration $\mathbb{P}\restrict{S} \ast \name{\mathbb{Q}\restrict{S}}$.
	 Define $\langle p, \name{q} \rangle$ to be an element of  $(\mathbb{P}\restrict{S} \ast \name{\mathbb{Q}\restrict{S}})^\bullet$, iff
	 \begin{enumerate}
	 	\item $\name{q} = \widehat{h}$ for some $h \in M$ with $\dom(h) \in [S]^{<\omega_1}$,
	 	\item $\supp(p) = \dom(h)$,
	 	\item for each $\alpha \in \dom(h)$ $\langle p_\alpha, \widehat{h(\alpha)} \rangle \in \mathbb{R}_\alpha^\bullet$. 
	 \end{enumerate}
	 \ed
	 
	\bc \label{suru1}  $(\mathbb{P}\restrict{S} \ast \name{\mathbb{Q}\restrict{S}})^\bullet$ is a dense subset of $\mathbb{P}\restrict{S} \ast \name{\mathbb{Q}\restrict{S}}$.
	\ec
	\bp Fix $\langle p, \name{q} \rangle \in \mathbb{P}\restrict{S} \ast \name{\mathbb{Q}\restrict{S}}$. Similarly to the proof of Lemma $\ref{hulyetechnikas}$, first recall that
	 \[ p \Vdash \ \name{q}  \in \name{\mathbb{Q} \restrict{S}} = \{ f: \ f \text{ is a function,} \dom(f) \in [S]^{<\omega_1}, (\forall \alpha \in \dom(f)) f(\alpha) \in \name{\mathbb{Q}_\alpha}  \} \]
	and $p \in \mathbb{P}\restrict{S}$, which is $\omega_1$-closed, as being then countable supported product of $\omega_1$-closed posets (Lemma $\ref{omclos}$). Now a suitable extension $p'$ of $p$ determines $\dom(\name{q}) \in [S]^{<\omega_1}$.
	Then we can extend $p'$ so that  for each $\alpha \in \dom(\name{q})$ it forces a value for $\supp(\name{q}(\alpha)) \in [X_\alpha]^{<\omega_1}$, and by picking a further extension $p'' \leq p'$ we can assume that $p''$ determines $(\name{q}(\alpha))(\beta) \in 2^{<\omega_1}$ for each $\alpha \in \dom(\name{q})$, $\beta \in \supp(\name{q}(\alpha))$. Therefore we obtain a function $h \in M$ with $\langle p'', \widehat{h} \rangle \in \mathbb{P}\restrict{S} \ast \name{\mathbb{Q}\restrict{S}}$, and 
	\[ \langle p'', \widehat{h} \rangle \leq \langle p, \name{q} \rangle.\]
	
	After a further extension we can assume that $\supp(p'') = \dom(h)$, and for each $\alpha \in \supp(p'')$ the tree $p''_\alpha$ has a top level, and the function $h(\alpha): \dom(h(\alpha)) \to p''_\alpha$ maps its domain into the top level of $p''_\alpha$.
	\ep

	\bc \label{suru2} $(\mathbb{P}\restrict{S} \ast \name{\mathbb{Q}\restrict{S}})^\bullet$ is isomorphic to $\mathbb{R}^\bullet\restrict{S}$.
	\ec
	\bp Simply assign to $r = \langle \langle p_\alpha, \widehat{g_\alpha} \rangle: \ \alpha \in S \rangle$ the pair $\langle \langle p_\alpha: \ \alpha \in S \rangle, \widehat{h} \rangle$, where $\dom(h) = \supp(r)$, and  $h(\alpha) = g_\alpha$ for each $\alpha \in \supp(r)$ .
	\ep
		 As $G \restrict{S} \cap \mathbb{R}^\bullet\restrict{S}$ is $\mathbb{R}^\bullet\restrict{S}$-generic over $M''$ (Lemma $\ref{surube}$), Claims $\ref{suru1}$, $\ref{suru2}$ (and applying Lemma $\ref{surube}$ again) imply that there are generic filters $F\restrict{S} \subseteq \mathbb{P}\restrict{S}$, $H\restrict{S} \subseteq \name{\mathbb{Q}\restrict{S}}[F\restrict{S}]$ such that
	 	\begin{equation} \label{fs} T \in M''[G\restrict{S}] = M''[F\restrict{S}][H\restrict{S}]. \end{equation}
	 	We could consider the model $M''[F\restrict{S}]$ in which the trees $T_\delta$ ($\delta \in S$) are already existing elements, but 
	 	at that moment
	 	we have not added the $|X_\delta|$ branches yet, implying that $2^{\omega_1}$ is small.
	 	
	 	\bd \label{MFS}
			In $M''[F\restrict{S}]$ we define
				\begin{equation} \label{Qdff} \mathbb{Q} = \name{\mathbb{Q}\restrict{S}}[F\restrict{S}] = \{ f \in \prod_{\delta \in S} \name{\mathbb{Q}_\delta}[F_\delta] : \ |\supp(f)|< \omega_1 \}, \end{equation}
		
			and
			
			\begin{equation} \label{Xdef}
				 X = \cup \{ X_\delta: \  \delta \in S \},
			\end{equation}
			Define $K \subseteq \mathbb{Q}$ be the filter so that
			\beeq \label{Hbov} M''[G\restrict{S}] = M''[F\restrict{S}][H\restrict{S}] = M''[F\restrict{S}][K] \eeq
			holds.
		\ed
	 
	 	We will have the following crucial lemma.
	 	\begin{lemma} \label{omcc}
	 		If $M' \subseteq M[G] $ is a c.t.m. such that $\mathbb{Q} \in M'$ (and $M \subseteq M'$) 
	 		then
	 		\[ M' \models \ \mathbb{Q} \text{ has the } \omega_2\text{-cc}. \]
	 	\end{lemma}
	 	\begin{proof}
	 		The proof is a straightforward application of the $\Delta$-system lemma, and $CH$ (which holds by Claim $\ref{sorozatt}$).
	 		Assume that $A = \{ f_\gamma: \ \gamma < \omega_2 \}$ is an antichain. Then, since the $f_\gamma$-s are countably supported $\eqref{Qdff}$, and $\omega_1^\omega = \omega_1 < \omega_2$ by $CH$, there is a subset $A' \subseteq A$ of size $\omega_2$ where
	 		\[ \{ \dom(f): \ f \in A' \} \text{ forms a }\Delta\text{-system with kernel }W. \]
	 		Now, since $x \in X_\delta \cap \dom(f)$ implies that $f(x) \in T_\delta$ (and 
	 		$|T_\delta|= \omega_1$)
	 		\[ |\{ f\restrict{W} :  \ f \in A' \}| \leq |\omega_1^W| = |\omega_1^\omega| = \omega_1. \]
	 		Thus one can find $\omega_2$-many elements of $A'$ such that any two of them coincide on $W$ (which is the intersection of their domains).
	 	\end{proof}

		Our next goal is to find a subset $Z \subseteq \cup \{ X_\delta: \ \delta \in S \}$,
		such that $|Z| \leq \omega_1$, and adding the branches indexed by the elements of $Z$ to $M''[F\restrict{S}]$ will result in a model that contains the tree $T$.


		We will see that adding the branches indexed by $Z  \cup \bigcup \{X_\delta: \ \delta \in S, \delta < \alpha \}$ will result a model which cannot contain all the branches $\mathcal{B}(T) \cap M[G]$, because there
		$2^{\omega_1}$ will not be large enough (i.e. in $M''[F\restrict{S}][K\restrict{{Z \cup (\bigcup\{X_\delta: \ \delta \in S, \delta < \alpha \})}}])$).
		From now on we will work in $M''[F\restrict{S}]$ to  prove that forcing with $\mathbb{Q}\restrict{Z  \cup \bigcup \{X_\delta: \ \delta \in S, \delta < \alpha \}}$
		will have these aforementioned properties.
		\begin{cl} \label{Zdff}
						There exists a set $Z \subseteq \cup \{ X_\delta: \ \delta \in S \}$ of size at most $\omega_1$, i.e.
						\[ M''[F\restrict{S}] \models |Z| \leq \omega_1 \]
						 such that
			\[ T \in M''[F\restrict{S}][K\restrict{Z}]. \]
		\end{cl}
		\begin{proof}
			Since $T \subseteq M''[F\restrict{S}]$ (in fact, $T \subseteq (2^{<\omega_1})^M = (2^{<\omega_1})^{M[G]}$, by Claim $\ref{sorozatt}$), and
			$T \in M''[F\restrict{S}][K]$,
			applying Lemma $\ref{nice}$ gives that there is a nice $\mathbb{Q}$-name $\sigma$ in $M''[F\restrict{S}]$ for a subset of $2^{<\omega_1}$, such that
			\[ \mathbb{1}_\mathbb{Q} \Vdash (\dot{T} \subseteq \widehat{2^{<\omega_1}}) \rightarrow (\sigma = T). \]
				$\sigma$ is a nice name i.e. is of the form
				\[ \sigma = \cup \{ \{ \widehat{f}  \}\times A_f : \ f \in 2^{<\omega_1} \}, \]
				where each $A_f \subseteq \mathbb{Q}$ is an antichain, and each $A_f$ is of size at most $\omega_1$ by Lemma $\ref{omcc}$.
				Let 
				\[ Z = \cup \{ \dom(a) : \ a \in A_f, f \in 2^{<\omega_1}  \} \subseteq X ,\]
				where $|Z| \leq |2^{<\omega_1}| \cdot |\omega_1| = \omega_1$.
			Then clearly $\sigma$ depends only on $K\restrict{Z}$, thus 
				\begin{equation}\label{Z} T \in M''[F\restrict{S}][K\restrict{Z}]. \end{equation}
						
		\end{proof}
		
		\bd
		Let $Y$ denote the set
		\begin{equation} \label{Ydefje} Y = Z \cup (\cup\{X_\delta: \ \delta \in S, \delta < \alpha \}). \end{equation}
		Obviously
		\[ T \in M''[F\restrict{S}][K\restrict{Y}],  \]
		let $N$ denote $M''[F\restrict{S}][K\restrict{Y}]$.
		\ed

		\begin{lemma} \label{hiany}
			$ N = M''[F\restrict{S}][K\restrict{Y}]$ contains $T$, but there are branches in $M[G]$ which are not contained in $N$.
		\end{lemma}
	
	\begin{proof}
		The next lemma is the key for verifying that $\mathcal{B}(T) \cap M[G] \supsetneq \mathcal{B}(T) \cap N$, where
		our assumption was that
		\beeq \label{Salpha} M[G] \models |\mathcal{B}(T)| \geq \omega_\alpha, \text{ and } \alpha \notin C,  \eeq
		thus $\alpha \notin S \subseteq C$.
		
		Now we have two cases depending on whether $\{ \delta: \ \delta \in S, \ \delta < \alpha \}$ is empty, or not. If the set $\{ \delta: \ \delta \in S, \ \delta < \alpha \}$ is empty, then since $\alpha \geq 2$ and $\alpha \notin C$ either $\alpha = 2$ holds, implying
		$2 \notin C$, or $\alpha>2$ thus $2 \notin S \subseteq  C$. Therefore Claims $\ref{szamolasb}$ and $\ref{cl1}$ will finish the proof of Lemma $\ref{hiany}$.
		\begin{cl} \label{szamolasb}
			
			If $2 \in C$, then
			\[ N= M''[F\restrict{S}][K\restrict{Y}] \models 2^{\omega_1}  < \omega_\alpha^{M[G]} . \]

		\end{cl}

		\begin{proof}
			First we will need that this case $GCH$ holds above $\omega_2= \omega_2^{M[G]}$ in $M''[F\restrict{S}]$. It suffices to prove the following claim.
			
			\bsc
			$M''[F\restrict{S}]$ can be obtained by a single forcing extension of $M$, where the poset has cardinality$^M$ less than $\omega_2^{M[G]}$.
			\esc
			\bp
			Since we defined $M''$ to be $M$ (Definition $\ref{M''}$), we have that $M''[F\restrict{S}]$ is a forcing extension of $M$, where we forced with the set $\mathbb{P}\restrict{S}$.
			In order to show that
			\[ M \models |\mathbb{P}\restrict{S}| < \omega_2^M, \]
			first, $\mathbb{P}_{\hom} \leq |2^{<\omega_1}|^\omega = \omega_1$ in $M$ (because $\mathbb{P}_{\hom} \subseteq [2^{<\omega_1}]^\omega $
			by  Definition $\ref{PhomD}$), hence $|\mathbb{P}\restrict{S}| \leq |\mathbb{P}_{\hom}|^\omega \cdot |\omega_1^{\omega} = \omega_1$. 
			\ep

			Recall that each $T_\delta$, given by the $\mathbb{P}_\delta$-generic filter $G_\delta$ is a subtree in $2^{<\omega_1}$ 
			of height $\omega_1$, thus is of size $\omega_1$.
			In $M''[F\restrict{S}]$ $\eqref{Qdff}$ and Definition $\ref{Qdef}$ (recalling that $F_\delta \subseteq \mathbb{P}_\delta$ is generic) give us that $\mathbb{Q}\restrict{Y}$ is of size 
			\begin{equation} \label{Qmeret} |\mathbb{Q}\restrict{Y}| = |Y|^\omega \cdot \omega_1^\omega. \end{equation}
			We have to determine $|Y|$. Let 
			\[ \sigma =  \sup (\{ \delta: \ \delta \in S, \delta < \alpha \} ) \geq 2. \]
			
			Since $S \in M''=M$, and by $\ref{Sclaim}$
			\beeq \label{M''S} M'' \models  |S| \leq \omega_1, \eeq
			which gives
			$\cf^{M''}(\sigma) \leq \omega_1$.
			Now we will show that $\sigma < \alpha$. Recall that $\alpha \notin C$ by our assumptions $\eqref{alfdef}$.
			
			\bsc
			\[ \sigma  \in C,  \]
			\[ \text{ in particular, } \sigma < \alpha. \]
			\esc
			\bp
			 First observe that as $M'' = M$, $\eqref{M''S}$ states that $\cf^{M}(\sigma) = \cf^{M''}(\sigma) \leq \omega_1$, and then the condition $\eqref{cofos}$ in Theorem $\ref{foo}$ implies that 
			\[ \sigma = \sup(S \cap \alpha) = \sup(C \cap \sigma) \in C. \]
		
			\ep

			Now $|Y| \leq \omega_1 + \sup \{|X_\delta|: \ \delta \in S, \delta < \alpha \}$, but $|X_\delta| = (\omega_\delta)^{M[G]}$ by Lemma $\ref{GCH0}$ (which is in fact $\omega_\delta$ of $M$), thus
			\begin{equation} \label{Ymerete}
				M''[F\restrict{S}] \models |Y| \leq \omega_{\sigma}^{M[G]} < \omega_\alpha^{M[G]} .
			\end{equation}
			Letting $\lambda$ denote $(\omega_\sigma)^{M[G]}$, note that $\lambda^+ = (\omega_{\sigma+1})^{M[G]}$,
			$\lambda^{++} = (\omega_{\sigma+2})^{M[G]}$ by Lemma $\ref{GCH0}$.
			Using $\eqref{Qmeret}$,
			\[ M''[F\restrict{S}] \models |\mathbb{Q}\restrict{Y}| 	\leq \lambda^\omega \cdot \omega_1^\omega. \]
			Recalling that $\mathbb{Q}$ is $\omega_2$-cc in $M''[F\restrict{S}]$ (Lemma $\ref{omcc}$), Lemma $\ref{hatvanyh}$ states that
			\begin{equation} \label{szam} M''[F\restrict{S}][K\restrict{Y}] \models 2^{\omega_1} \leq ((\lambda^\omega \cdot \omega_1^\omega)^{\omega_1})^{M''[F\restrict{S}]}. \end{equation}
			For calculating this cardinal in $M''[F\restrict{S}]$ we have two cases.
			
			\begin{itemize}
				\item If $\sigma$ is limit (and thus $\cf^{M''[F\restrict{S}]}(\sigma) \leq \omega_1$), then first recall that no cofinalities were collapsed in our case. Using the conditions for $C$ in Theorem $\ref{foo}$ we have that
				$\sigma +1 \in C$, therefore for $\alpha \notin C$ $\alpha \geq \sigma + 2$. This case using that  $\sigma \geq 2$,
				and by the $GCH$ in $M''[F\restrict{S}]$ above $\omega_2^{M[G]}$ for $\lambda = \omega_\sigma^{M[G]}$
				
				\[ M''[F\restrict{S}] \models (\lambda^\omega \cdot \omega_1^\omega)^{\omega_1} \leq 2^\lambda = \lambda^+ = \omega_{\sigma+1}^{M[G]} < \omega_{\alpha}^{M[G]}.    \]
				\item If  $\sigma$ is a successor, then $\cf(\lambda) = \lambda > \omega_1$.
				Hence using again that $GCH$ holds above $\omega_2^{M[G]}$ (and that $\omega_2^{M[G]} \leq \omega_\sigma^{M[G]} = \lambda$)
				\[ M''[F_S] \models (\lambda^\omega \cdot \omega_1^\omega)^{\omega_1} = \sup \{ \beta^{\omega_1} : \beta < \lambda \} = \lambda <  \omega_{\alpha}^{M[G]}. \]
			\end{itemize}
			We get that $\eqref{szam}$ and the above estimations give
			\[ M''[F\restrict{S}][K\restrict{Y}] \models 2^{\omega_1} < \omega_\alpha^{M[G]}. \]
			
		\end{proof}
		
		\bc \label{cl1}
		If $2 \notin C$
		\[  M''[F\restrict{S}][K\restrict{Z}] \models 2^{\omega_1} < \kappa = \omega_2^{M[G]}\]
		\ec
		\begin{proof}
			First recall that 
			\[ M''[F_S] \models |Z| = \omega_1, \]
			
			Since each $T_\delta$ is of size $\omega_1$,
			\[ M''[F\restrict{S}] \models |\mathbb{Q}\restrict{Z}| = |Z|^\omega \cdot \omega_1^\omega = \omega_1^\omega = \omega_1 \]
			by $CH$.
			Then $\mathbb{Q}\restrict{Z}$ is trivially $\omega_2$-cc, and by Lemma $\eqref{hatvanyh}$.
			\[ M''[F\restrict{S}][K\restrict{Z}] \models 2^{\omega_1} = (\omega_1^{\omega_1})^{M''[F\restrict{S}]}. \]
			We can calculate $(\omega_1^{\omega_1})^{M''[F\restrict{S}]}$, because by  Lemma $\ref{GCH}$ (with $M' = M''[F\restrict{S}]$) $\kappa$ is inaccessible in $M''[F\restrict{S}]$, therefore
			
			\[ M''[F\restrict{S}] \models \omega_1^{\omega_1} < \kappa = \omega_2^{M[G]}, \]
			as desired.
		\end{proof}
		This finishes the proof of Lemma $\ref{hiany}$, since it follows from
		$\{ \delta: \ \delta \in S, \ \delta < \alpha \} = \emptyset$ that $Y = Z$, and $\omega_2^{M[G]} < \omega_\alpha^M[G]$.
		
	\end{proof}

	 	Next we prove that $M''[F\restrict{S}][K] = M''[F\restrict{S}][K\restrict{Y}][K\restrict{X \setminus Y}] = N[K\restrict{X \setminus Y}]$ will contain each branch of $T$ from the final model $M[G]$.
	 	\begin{cl} \label{osszesag}
	 		\[ M[G] \cap \mathcal{B}(T) = M''[F\restrict{S}][K] \cap \mathcal{B}(T) \]
	 	\end{cl}
	 	\begin{proof}
	 		First, $M[G]$ is either  $M[G\restrict{C}]$, or $M[I][G\restrict{C}]$, and $G\restrict{C} \simeq G\restrict{S} \times G\restrict{C \setminus S}$, 
	 		$I \simeq I\restrict{\mu} \times I\restrict{\kappa \setminus \mu}$, and $M''$ is either  $M$, or $M[I\restrict{\mu}]$.
	 		Also recall that
	 		$M''[G\restrict{S}] = M''[F\restrict{S}][K]$ by $\eqref{Hbov}$ from Definition $\ref{Qdff}$.
	 		This means that
	 		\[ M''[F\restrict{S}][K][G\restrict{C \setminus S}] = M''[G\restrict{S}][G\restrict{C \setminus S}] = M''[G\restrict{C}], \]
	 		and our final model $M[G]$ is either $M''[G\restrict{C}]$, or $M''[G\restrict{C}][I\restrict{\kappa \setminus \mu}]$.
	 		
	 		For this forcing extensions (i.e. extension by $G\restrict{C \setminus S}$ and $I\restrict{\kappa \setminus \mu}$) we would like to apply Lemmas $\ref{nincsuj}$, $\ref{nincsuj3}$ to ensure that none 
	 		of them add new branches to $T$.
	 		First, the $\mathbb{R}\restrict{C \setminus S}$-generic filter $G\restrict{C \setminus S}$ doesn't add new branches by Lemma $\ref{nincsuj3}$ (applying with $M' = M''[F\restrict{S}][K]$). Second, $\mathbb{L}\restrict{\kappa \setminus \mu} \in M$ is
	 		$\omega_1$-closed in $M$. But there are no new $\omega$-sequences in $M''[G\restrict{C}] \subseteq M[G]$ by Corollary $\ref{sorozatt}$,
	 		thus it can be easily seen that
	 		\[ M''[G\restrict{C}] \models \mathbb{L}\restrict{\kappa \setminus \mu} \text{ is } \omega_1\text{-closed.} \]
	 		Now we can apply Lemma $\ref{nincsuj}$, thus forcing with the $\omega_1$-closed
	 		$\mathbb{L}\restrict{\kappa \setminus \mu}$ adds no new branches to $T$.
	 	\end{proof}
	 	
	 	As we got that all branches of $T$ are contained in $M''[F\restrict{S}][K]=M''[F\restrict{S}][K\restrict{Y}][K\restrict{X \setminus Y}]$, and $T \in N = M''[F\restrict{S}][K\restrict{Y}] $, it remains to show that 
	 	the extension of $N$ with the filter $K\restrict{X \setminus Y}$ adds more than $\omega_\alpha^{M[G]}$-many branches,
	 	which would contradict $\eqref{alphadef}$.
	 	
	 	\begin{cl}
	 		There exists an ordinal $\gamma \in S$ such that $\gamma > \alpha$.
	 	\end{cl}
	 	\begin{proof}
	 			By the definition of $X,Y,Z$ $\eqref{Xdef}$, $\eqref{Ydefje}$ we have
	 			\[ Y = Z \cup \bigcup\{ X_\delta: \ \delta \in S \setminus \alpha \}, \ X = \bigcup\{ S_\delta: \delta \in S \}, \]
	 			and
	 			\[ X \setminus Y \subseteq  \cup \{X_\delta \setminus Z : \ \delta \in S, \ \delta > \alpha \} \]
	 			and we know that 
	 			\[ M''[F\restrict{S}] \models |Z| \leq \omega_1 < \omega_2^{M[G]}. \]
		 	 Now assume on the contrary that $S \subseteq \alpha+1$ (hence also $S \subseteq \alpha$), then
		 	 the equality $X=Y$ would hold, which contradicts to the fact  
		 	 that extension with $K\restrict{X \setminus Y}$ adds branches to $T$ (because of Lemma $\ref{hiany}$).
	 	\end{proof}

	 	Therefore the fact that $|X_\delta| = \omega_\delta^{M[G]}$ for $\delta \in C$ (by  Lemma $\ref{GCH}$), and
	 	\[ M''[F\restrict{S}] \models |Z| \leq \omega_1, \]
	 	 (together with  $0,1 \notin S$) implies that for every $\delta \in C$
	 	 \beeq \label{Xdelta'} |X_\delta \setminus Z| =|X_\delta| = \omega_\delta^{M[G]}. \eeq
	 	 
	 	 \bd
	 	 Let $X_\delta \setminus Z$ be denoted by $X'_\delta$, and let
	 	 \[ X' = X \setminus Y = \cup \{X'_\delta: \ \delta \in S: \delta > \alpha \}. \]
	 	 Then $\eqref{Xdelta'}$ states 
	 	 \[  |X'_\delta| = \omega_\delta^{M[G]} \ \ (\delta > \alpha). \]
	 	 \ed
	 	\begin{cl} \label{xkv}
			There is a 	set $X'' \in N$, $X'' \subseteq X'$ of size $\leq \omega_\alpha^{M[G]}$ such that decomposing $\mathbb{Q}\restrict{X'}$
			into the product $\mathbb{Q}\restrict{X''} \times \mathbb{Q}\restrict{X' \setminus X''}$, and thus obtaining the filters $K\restrict{X''}$, 
			$K\restrict{X' \setminus X''}$,
			\[ N[K\restrict{X''}] \cap \mathcal{B}(T) = M[G] \cap \mathcal{B}(T), \]
			that is, all branches of $T$ are in the model  $N[K\restrict{X''}]$.
	 		
	 	\end{cl}
	 	\begin{proof}
		 	 In $M''[F\restrict{S}][K]$ (and in $M[G]$)  there are $\omega_\alpha^{M[G]}$ branches of $T$ (by Claim $\ref{osszesag}$, and $\eqref{alfdef}$),
		  and since each branch corresponds to a function from $\omega_1$ to $2$, the set of branches can be identified with a
		  function $\omega_\alpha^{M[G]} \times \omega_1 \to 2$, which can be identified with a subset of $\omega_\alpha^{M[G]} \times \omega_1$. Let $B \subseteq \omega_\alpha^{M[G]} \times \omega_1$
		  be the set in $N[K\restrict{X \setminus Y}]$ coding the branches of $T$. Then using Lemma $\ref{nice}$,
		  there is a nice $\mathbb{Q}\restrict{X \setminus Y}$-name $\sigma$ in $N$ such that
		  \[ \mathbb{1}_{\mathbb{Q}\restrict{X \setminus Y}} \Vdash (\dot{B} \subseteq \omega_\alpha^{M[G]} \times \omega_1) \rightarrow (\sigma = \dot{B}) \]
		  (where $\dot{B}$ is a $\mathbb{Q}\restrict{X \setminus Y}$-name in $N$ for $B$).
		  Now if
		  \[ \sigma = \{ \widehat{\langle \mu, \nu \rangle} \times A_{\langle \mu, \nu \rangle} : \ \mu < \omega_\alpha^{M[G]}, \nu <\omega_1 \} ,\]
		  where each $A_{\langle \mu, \nu \rangle} \subseteq \mathbb{Q}\restrict{X \setminus Y}$ is an antichain, we have 
		  		   (because by Lemma $\ref{omcc}$ 
		   \[ N \models \mathbb{Q}\restrict{X \setminus Y} \text{  is }\omega_2\text{-cc}), \]
		  that  $|A_{\langle \mu, \nu \rangle}| \leq \omega_1$ for each $\mu$ $\nu$.  
		   Define 
		  \[ X'' = \cup \{ \dom(a) : a \in A_{\langle \mu, \nu \rangle}, \ \mu < \omega_\alpha^{M[G]}, \nu <\omega_1 \} \subseteq X \setminus Y, \]
		  and note that 
		  \[ |X''| \leq \omega_1 \cdot \omega_\alpha^{M[G]} = \omega_\alpha^{M[G]}.\]
		  Then clearly $B = \sigma[K\restrict{X \setminus Y}]$, and for any $\mathbb{Q}\restrict{X \setminus Y}$-generic filter $J$
		  \[ \sigma[J] \in N[J\restrict{X''}]. \]
		 	This completes the proof the Claim.
	 	\end{proof}
	 	
	 	\begin{cl} \label{izomc}
	 		(Let $X''$ be given by Claim $\ref{xkv}$.) There is a set $X''' \subseteq X' \setminus X''$ such that (in $N$) $\mathbb{Q}\restrict{X''}$ is
	 		isomorphic to $\mathbb{Q}\restrict{X'''}$. (This gives rise to the split up of $\mathbb{Q}\restrict{X'}$ to the product 
	 		\[ \mathbb{Q}\restrict{X''} \times \mathbb{Q}\restrict{X'''} \times \mathbb{Q}\restrict{X' \setminus(X'' \cup X''')}). \]
	 	\end{cl}
		 \begin{proof}
		 	Since $X''$ is given by Lemma $\ref{xkv}$ $|X''| \leq \omega_\alpha^{M[G]}$, and recall that
		 	\[ \mathbb{Q}\restrict{X''} = \{ f: \ \dom(f) \subseteq X'' \text{ is countable,} \ (x \in X_\delta) \rightarrow (f(x) \in T_\delta) \}, \]
		 	and 
		 		 	\[ |Y| < \omega_{\alpha}^{M[G]} \]
		 		 	by $\eqref{Ymerete}$.
		 	For each $\delta > \alpha$, $\delta \in S$,
		 	using the facts
		 	\[ |X_\delta| = |X'_\delta| = \omega_\delta^{M[G]}, \] 
		 	\[ |X''| \leq \omega_\alpha^{M[G]} \]
		 	there is a set  $X'''_\delta \subseteq X'_\delta \setminus X'$		 of size $|X'' \cap X_\delta|$.
		 Letting $X''' = \cup \{ X'''_\delta: \ \delta > \alpha, \delta \in S \}$, it follows that for each $\delta$
		 \[ N \models  |X''' \cap X_\delta| = |X'' \cap X_\delta|,\]
		 therefore (by $\eqref{Qdff}$, and the absoluteness of this definition, Lemma $\ref{sorozatt}$) 
		 \[ N \models \mathbb{Q}\restrict{X''} \simeq \mathbb{Q}\restrict{X'''}, \]
		  as desired.
		 \end{proof}

		The next point will be the key in the proof of Lemma $\ref{alfas}$, where we will make use of the homogeneity of the $T_\delta$-s similarly as in \cite{Jin}.
		
		\begin{lemma} \label{b'}
			Whenever $b \in N$ is a  $\mathbb{Q}\restrict{X''}$-name such that $b[K\restrict{X''}] \in N[K\restrict{X''}]$ is a branch through $T$, 
			$b[K\restrict{X''}] \notin N$, then there exists a $\mathbb{Q}\restrict{X''}$-name $b'$ such that (in $N$)
			\[  \mathbb{1}_{\mathbb{Q}\restrict{X''}} \Vdash \ b' \text{ is a new branch in } T.    \]
		\end{lemma}
		For the proof we will need the following claim.
			
			\begin{cl} \label{izomcl}
				 For any fixed element $q \in \mathbb{Q}\restrict{X''}$, and  filter $J$ which is $\mathbb{Q}\restrict{X''}$-generic 
					 over $N$, there is another $\mathbb{Q}\restrict{X''}$-generic filter $J'$ (over $N$)  such that $q \in J'$, and
					 \[ J' \in N[J]. \]
			\end{cl}
			\begin{proof}
				Recall that each $T_\delta$ is a homogeneous normal tree given by forcing with $\mathbb{P}_\delta$,
				normality implies that for each $t \in T_\delta$, $\htt(t,T_\delta) < \gamma < \omega_1$, there exists an element $t_\gamma \in T_\delta \cap 2^\gamma$ with  $t \subseteq t_\gamma$. Now by the definition of $\mathbb{Q}\restrict{X''}$  $\eqref{Qdff}$, 
				using $\dom(q)$-s countability, we can assume
				that there exists a countable ordinal $\varrho$ such that
				 $x \in \dom(q)$ implies that $q(x) \in 2^\varrho$.
				 
				Let $d = \dom(q) \subseteq X''$. 
				Let $u \in J$ be such that $\dom(u) = d$, and for each $x \in d$ $u(x) \in 2^\varrho$. For each  $\delta \in S$, $x \in d \cap X_\delta$  we define 
				\begin{equation} \label{uq}  F_x = F_{u(x)q(x)}: 2^{\leq \omega_1} \to 2^{\leq \omega_1}, \end{equation}
				which is an automorphism of $T_\delta$ if we restrict to it, because $T_\delta$ is homogeneous.
				Now we define an automorphism $\varphi$ of the poset $\mathbb{Q}\restrict{X''}$.
				\[ \varphi: \mathbb{Q}\restrict{X''} \to \mathbb{Q}\restrict{X''}, \]
				\[ f \mapsto  \varphi(f), \]
				such that
				\begin{equation} \label{fuq} (\varphi(f))(x) = \begin{cases}  F_x(f(x))  \ \text{ if } x \in d \\
							f(x) \ \ \  \text{ otherwise,} \end{cases}
				\end{equation}
				i.e. (considering $\mathbb{Q}\restrict{X''}$ as the countable support product of $T_\delta$-s) we applied
				an automorphism on some coordinates (coordinates in $d$). 
				Obviously $\varphi(u) = q$, by $\eqref{uq}$ and $\eqref{fuq}$.
				It is straightforward to check that $\varphi$ is indeed an automorphism, since for a pair $q_1,q_2 \in \mathbb{Q}\restrict{X''}$
				\[ q_1 \leq q_2 \ \iff \ (\forall x \in \dom(q_1) \cap \dom(q_2)): (q_1(x) \supseteq q_2(x). \]
				Now letting 
				\begin{equation} J' = \varphi[J] = \{f: \ f= \varphi(g) \text{ for some }g \in J \} \in  N[J] \end{equation}
				we obtain a filter containing $\varphi(u) = q$.
				It remained to check that $J'$ is generic over $N$.
				Suppose that $D \in N$ is a dense subset of $\mathbb{Q}\restrict{X''}$. Then
				$\varphi^{-1}(D)$ is also a dense subset (note that $\varphi \in N$), 
				 and if $q_D \in J \cap \varphi^{-1}(D)$, then $\varphi(q_D) \in J'$ and $\varphi(q_D) \in D$.
				
			\end{proof}
		\begin{proof}(Lemma $\ref{b'}$)
			Now suppose that $b[K\restrict{X''}]$ is a new branch of $T$, and $q \in K\restrict{X''}$ forces that, i.e.
			 \begin{equation} \label{qforszolja} q \Vdash \ b \text{ is a new branch in } T. \end{equation}
			 			Now using Claim $\ref{izomcl}$, for any filter $J$ which is generic over $N$, there exists a generic filter $J'$ containing $q$, such that
			\[ J' \in N[J]. \]
			This means that 
			\[ N[J'] \subseteq N[J], \]
			and by $\eqref{qforszolja}$ and $q \in J'$,
			\[ N[J'] \models b[J'] \text{ is a new branch in } T, \]
			thus
			\[ N[J] \models b[J'] \text{ is a new branch in } T. \]
			Since $J$ was arbitrary, we have that
			\[ \mathbb{1}_{\mathbb{Q}\restrict{X''}} \Vdash \ \exists b' \text{ a new branch in } T.    \]
			Finally, applying the maximal principle \cite[II., Thm. 8.2]{kunen}, there exists a name $b' \in N$ such that
			\[ \mathbb{1}_{\mathbb{Q}\restrict{X''}} \Vdash \ b' \text{ a new branch in } T.    \]
		\end{proof}
		
		Now, if $\psi: \mathbb{Q}\restrict{X''} \to \mathbb{Q}\restrict{X'''}$ is an isomorphism (provided by Claim $\ref{izomc}$), 
		(and $\psi^*$ denotes the induced operation between the $ \mathbb{Q}\restrict{X''}$-names and $ \mathbb{Q}\restrict{X'''}$-names)  then by our previous lemma clearly
		\[ N \models \ \left( \mathbb{1}_{\mathbb{Q}\restrict{X'''}} \Vdash \ \psi^*(b') \text{ a new branch in } T \right). \]

		The next Lemma completes the proof of Lemma $\ref{alfas}$, since Lemma $\ref{xkv}$ guarantees that $N[K\restrict{X''}]$ contains all branches of $T$.
		\begin{lemma} \label{utso} 
			For $b'$ given by Lemma $\ref{b'}$
			 \[ N[K\restrict{X''}] \models  \ `` \mathbb{1}_{\mathbb{Q}\restrict{X'''}} \Vdash \ \psi^*(b') \text{ is a new branch in } T" .  \]
		\end{lemma}
		Before the proof recall the facts implying that Lemma $\ref{utso}$ completes the proof of $\ref{alfas}$. By Lemma $\ref{xkv}$ each branch of $T$ which is in $M[G]$ already appears in $N[K\restrict{X''}]$, that is
		\[ \mathcal{B}(T) \cap M[G] = \mathcal{B}(T) \cap N[K\restrict{X''}]. \]
		But $\psi^*(b')[K\restrict{X'''}]$ is a new branch in $N[K\restrict{X''}][K\restrict{X'''}]$ which is a model between 
		$M$ and $M[G]$, a contradiction.
		\begin{proof}
			Whenever $J \subseteq \mathbb{Q}\restrict{X'''}$ is generic over $N$ we have that
			 $\psi^*(b')[J] \in 2^{\omega_1}$ is a new branch of $T$ (i.e. not in $N$), therefore we claim the following.
			\begin{cl} \label{elsomodell}
				\beeq \label{elsomeq} \forall q \in \mathbb{Q}\restrict{X'''}  \exists q_0,q_1 \leq q, \delta <\omega_1: \ \left(N \models \ q_i \Vdash (\psi^*(b'))(\delta)=i  \right) \ (i=0,1) \eeq
			\end{cl}
			\begin{proof}
				Assume on the contrary, that $q$ is a counterexample. But then for each $\delta <\omega_1$ there exists $i_\delta \in \{0,1 \}$ such that  whenever $q' \leq q$ decides $\psi^*(b')(\delta)$, then
				\[  q' \Vdash (\psi^*(b'))(\delta)= i_\delta,  \]
				from which
				\[ q \Vdash (\psi^*(b'))(\delta)= i_\delta. \]
				Now, since we defined $\langle i_\delta: \ \delta < \omega_1 \rangle$ in $N$, $q$ 
				determines $(\psi^*(b'))(\delta)$, thus forces that $\psi^*(b')$ is not a new branch, a contradiction.
			\end{proof}
			But this claim is true even in $N[K\restrict{X''}]$.
			\begin{cl} \label{vegso}
				\[ \forall q \in \mathbb{Q}\restrict{X'''}  \exists q_0,q_1 \leq q, \delta <\omega_1: \ \left(N[K\restrict{X''}] \models q_i \Vdash (\psi^*(b'))(\delta)=i  \right) \  (i=0,1) \]
			\end{cl}
			\begin{proof}
				For a fixed $q$ let $q_0,q_1$ and $\alpha$ given by Claim $\ref{elsomodell}$ for which $\eqref{elsomeq}$ holds. 
					Let $J \subseteq \mathbb{Q}\restrict{X'''}$ be generic over $N[K\restrict{X''}]$. If $q_i \in J$ for some $i \in 2$,
					then $J$ is also generic over $N$, and by Lemma $\ref{elsomodell}$, 
					\begin{equation} \label{absz}  N[J] \models  \psi^*(b'[J](\delta)=i. \end{equation}
				Statement $\eqref{absz}$ is absolute between transitive models,
					thus
					\[  N[K\restrict{X''}][J] \models  \psi^*(b'[J](\delta)=i. \] 
					We conclude that choosing the same $q_i$-s and $\alpha$ works.
			\end{proof}
			Let $B \in N[K\restrict{X''}]$ denote the set of branches of $T$ in $N[K\restrict{X''}]$.
			Now if $q \in \mathbb{Q}\restrict{X'''}$ forces that $\psi^*(b')$ is not a new branch
			\[   N[K\restrict{X''}] \models  \ \left( q \Vdash \ \text{ the branch }\psi^*(b') \text{ is in } B \right),  \]
			then in a fixed generic filter $K\restrict{X'''}$ containing $q$
			\[   N[K\restrict{X''}][K\restrict{X'''}] \models  \psi^*(b)[K\restrict{X'''}] \in B. \]
			
			This implies that there exists a branch $b_0 \in B$, such that $b_0 = \psi^*(b')[K_{X'''})$, which is forced by some $q' \in K\restrict{X'''}$, $q' \leq q$
			\[  N[K\restrict{X''}] \models \left( \mathbb{1}_{\mathbb{Q}\restrict{X'''}} \Vdash \widehat{b_0} =  \psi^*(b') \right), \]
			i.e. $q'$ determines $\psi^*(b')$, which contradicts to Claim $\ref{vegso}$.
			
		\end{proof}
		
\begin{rem}
	By further forcing we could prescribe $2^{\omega_1}$ to be any cardinal greater than or equal to the cardinalities of branches of Kurepa trees.
\end{rem}

\begin{question}
 In Theorem $\ref{foo}$ can we drop the condition that for $\alpha \in C$ if 
 \[ \omega \leq \cf(\alpha) \leq  \omega_1, \] 
 then $\alpha +1$ must be contained in $C$? Or
is it true that the existence of a Kurepa tree with $\omega_\alpha$-many branches (with $\cf(\alpha)$ either countable or $\omega_1$) implies
not only the inequality $2^{\omega_1} \geq \omega_{\alpha+1}$ (Konig's inequality), but the existence of a Kurepa tree with exactly $\omega_{\alpha+1}$ branches?
\end{question} 
 

\section*{Acknowledgment}
I would like to thank my masters thesis supervisor Péter Komjáth for presenting me the problem, and also giving hints and ideas, providing some useful advice.

I also gratefully thank to the referee for  constructive comments and recommendations, also pointing to possible simplifications of otherwise tedious technical arguments.

\end{document}